\documentclass[review,hidelinks,onefignum,onetabnum]{siamart220329}

\usepackage{booktabs}  
\usepackage{array}       
\usepackage{booktabs}    
\usepackage{caption}     
\usepackage{geometry}   
\usepackage{graphicx}
\usepackage{subcaption} 
\usepackage{rotating}
\usepackage{float} 
\usepackage[utf8]{inputenc}

\usepackage{lipsum}
\usepackage{amsfonts}
\usepackage{graphicx}
\usepackage{epstopdf}
\usepackage{algorithmic}
\usepackage{amsfonts}
\usepackage{amsmath,booktabs,ctable,threeparttable}
\usepackage{amssymb,amsfonts,boxedminipage}
\usepackage{algorithm}
\usepackage{algorithmic}

\ifpdf
  \DeclareGraphicsExtensions{.eps,.pdf,.png,.jpg}
\else
  \DeclareGraphicsExtensions{.eps}
\fi

\newsiamremark{remark}{Remark}
\newsiamremark{hypothesis}{Hypothesis}
\crefname{hypothesis}{Hypothesis}{Hypotheses}
\newsiamthm{claim}{Claim}
\nolinenumbers

\headers{A numerical analysis of sLS problems}{ZHONGXIAO JIA AND XINYUAN WAN}

\title{A Numerical Analysis of Sketched Linear Squares Problems and Stopping Criteria for Iterative Solvers\thanks{Submitted to the editors DATE.
\funding{Supported in part by the National Natural Science Foundation of China under grants
12171273 and 12571404.
}}}

\author{ZHONGXIAO JIA\thanks{Corresponding author. Department of Mathematical Sciences, Tsinghua University, 100084 Beijing, China
  (\email{ jiazx@tsinghua.edu.cn}).}
\and XINYUAN WAN\thanks{Department of Mathematical Sciences, Tsinghua University, 100084 Beijing, China
  (\email{wanxy22@mails.tsinghua.edu.cn}).}}

\usepackage{amsopn}

\ifpdf
\hypersetup{
  pdftitle={A Numerical Analysis of Sketched Linear Squares Problems and Stopping Criteria for Iterative Solvers},
  pdfauthor={ZHONGXIAO JIA AND XINYUAN WAN}
}
\fi

\begin{document}

\maketitle
\nolinenumbers

\begin{abstract}
Randomized subspace embedding methods have had a great impact on
the solution of a linear least
squares (LS) problem by reducing its row dimension, leading to a randomized or
sketched LS (sLS) problem, and use the solution of the sLS problem as an approximate
solution of the LS problem. This work makes a
numerical analysis on the sLS problem,
establishes its numerous theoretical properties, and
show their crucial roles on the most effective and efficient use of iterative solvers.
We first establish a compact bound on the
norm of the residual difference between the solutions of the LS and sLS problems,
which is the first key result towards understanding the rationale of the sLS problem.
Then from the perspective of backward errors, we prove that the solution of the
sLS problem is the
one of a certain perturbed LS problem with minimal backward error,
and quantify how the embedded quality affects the residuals,
solution errors, and the
relative residual norms of normal equations of the LS and sLS problems. These theoretical results enable us to propose new novel and reliable general-purpose
stopping criteria for iterative solvers for the
sLS problem, which dynamically monitor stabilization
patterns of iterative solvers for the LS problem itself
and terminate them at the earliest iteration.
Numerical experiments justify the theoretical bounds and demonstrate that the new stopping
criteria work reliably and result in a tremendous reduction in computational cost
without sacrificing attainable accuracy.
\end{abstract}


\section{Introduction}
The solution of large-scale linear least squares (LS) problem:
\begin{equation}\label{eq:original problem}
    x_{ls}=\arg\min_x \|Ax - b\|,
\end{equation}
is a fundamental task in scientific computing, where $A \in \mathbb{R}^{m \times n}$ with
$m \gg n$ is of full column rank, $b \in \mathbb{R}^m$, and $\| \cdot \|$
is the Euclidean norm of a vector and the spectral norm of a matrix.  We denote by
$r_{ls}=Ax_{ls}-b$ the minimum residual of the problem \eqref{eq:original problem}.
The introduction and development of randomized algorithms that leverage dimensionality reduction techniques may accelerate traditional deterministic LS solvers substantially. Among these, subspace embedding methods 
have emerged as powerful tools for reducing the dimensionality of the LS problem, as shown below.

A matrix \(S \in \mathbb{R}^{d \times m}\) with $n\leq d<m$ is called a subspace embedding if, for all
\(Ax-b \in \mathcal{L}\subseteq \mathbb{R}^m\), it satisfies
\begin{equation}\label{eq:basic_inequality}
(1 - \epsilon) \|Ax - b\|^2 \leq \|S(Ax - b)\|^2 \leq (1 + \epsilon) \|Ax - b\|^2,
\end{equation}
with high probability, where \(\epsilon \in (0, 1)\) is the embedding error parameter and is called distortion \cite{4031351}, and $\mathcal{L}$ is a given subspace.
It states that
$S$ preserves the geometric structures of $\|Ax-b\|$ over $\mathcal{L}$ with the
distortion $\epsilon$ and
the set $\mathcal{L}$ of higher $m$-dimensional data is embedded into a set
of lower $d$-dimensional data in $\mathbb{R}^d$ with approximately pairwise
equal lengths. Such subspace embedding can be
used to transform the LS problem
\eqref{eq:original problem} into a low dimensional surrogate, i.e., the sLS
problem:
\begin{equation}\label{eq:sketched problem}
 x_s=\arg\min_x \|S(Ax - b)\|,
\end{equation}
in which we assume that $SA$ is of full column rank. If $n\leq d\ll m$, the row dimension $d$
of the sLS problem \eqref{eq:sketched problem} is greatly reduced relative to the row dimension $m$
of the LS problem \eqref{eq:original problem}.
Throughout, we denote by \(r_{s}=Ax_{s}-b\) the residual of the solution $x_s$ to the sLS problem.

As delineated in \cite{Martinsson_Tropp_2020}, upon obtaining the sLS problem \eqref{eq:sketched
problem}, the methodological approaches for its resolution can be categorized into three
classes: The first class is on numerical solution of the sLS problem
\eqref{eq:sketched problem}, referred to as
the sketch-and-solve paradigm \cite{doi:10.1137/15M1040827}, and uses $x_s$ as an approximate
solution of the LS problem \eqref{eq:original problem}. It
has recently been used to solve a sequence of the LS problems
produced by the GMRES method at iterations \cite{tropp2024}. The
second class
leverages \eqref{eq:sketched problem} to obtain a right preconditioner from \eqref{eq:sketched
problem} for the LS problem \eqref{eq:original problem} \cite{derezinski2025faster,epperly2024,epperly2025,rokhlin2008fast,shan2025optimal}, referred to as the
sketch-and-precondition paradigm. The third class uses
sketching to improve solution accuracy of the sketch-and-solve paradigm by repeatedly using
it to reduce the residual vector \cite{xu2024randomized}; there is a similar approach for the constrained LS problems in the convex
optimization \cite{pilanci2016iterative}.


In this paper we consider the sketch-and-solve paradigm.
Its key theoretical foundation, which originates in \cite{4031351} and has gained substantial developments in \cite{Martinsson_Tropp_2020,meier2024randomized,woodruff2014sketching}, is
the inequality
\begin{equation}\label{eq:basic foundation}
    \|r_{ls}\|\leq\|r_s\| \leq \left(\frac{1 + \epsilon}{1 - \epsilon}\right)^{\frac{1}{2}}\|r_{ls}\|.
\end{equation}
Based on \eqref{eq:basic foundation}, one usually exploits Krylov iterative solvers, e.g., the
commonly used LSQR and LSMR algorithms
\cite{doi:10.1137/1.9781611971484,fongsisc11}, to compute a $k$-step
approximate solution $x_k^s$ of the sLS problem \eqref{eq:sketched problem} at iteration $k$ and
use it as an approximate solution of the original LS problem \eqref{eq:original problem}.

A deficiency in the current literature lies in
its reliance solely on \eqref{eq:basic foundation} to
establish connections between problem \eqref{eq:original problem} and problem \eqref{eq:sketched problem}.
This approach provides only a partial perspective by exclusively focusing on residual magnitudes but
neglecting some other crucial aspects: the more important
directional difference $r_{ls}-r_s$ and more fundamental
intrinsic connections underlying these two problems, e.g., relationships between
$x_{ls}$ and $x_s$, iterative solutions of \eqref{eq:original problem} and \eqref{eq:sketched
problem} and their residuals, and so forth. In
fact, only two comparable but {\em not} small residuals $r_{ls}$ and $r_s$ in size
shown by \eqref{eq:basic foundation} tell us very little, as we shall see in this paper. Such inadequacy
has consequently led to theoretically unsound stopping criteria when applying iterative methods to solve the
sLS problem \eqref{eq:sketched problem}, particularly in terms of directional consistency and error
propagation mechanisms. Overall, the sketch-and-solve paradigm lacks the integration of randomization into
classical numerical algorithms and rigorous perturbation analysis on the sLS
problem \eqref{eq:sketched problem}, and does not establish necessary bounds on numerical error propagations.
Such absence leaves some critical gaps between the sLS problem and the original LS problem when
considering the rationale and working mechanism of the sketch-and-solve methods, particularly the reliable
design of general-purpose stopping criteria for iterative solvers. In this work, we are concerned
with the size of $\|r_{ls}-r_s\|$ in terms of $\epsilon$, and
make a comprehensive and rigorous backward error analysis and perturbation analysis. The results will
explicitly quantify how approximation errors of the sLS problem \eqref{eq:sketched
problem} interact with traditional numerical
uncertainties;when using the $x_s$ of \eqref{eq:sketched
problem} as an approximate solution of \eqref{eq:sketched
problem}, they enable us to design the most effective and efficient stopping criteria for iterative
solvers for \eqref{eq:sketched problem} that, unlike traditional ones,
terminate an
iterative solver at right time without sacrificing
ultimately attainable accuracy.

We shall first establish a compact bound for the residual error $\|r_{ls}-r_s\|$ in terms of $\epsilon$.
Then from the perspectives of backward error analysis and perturbation analysis, we investigate the
sLS problem \eqref{eq:sketched problem} for a given subspace embedding matrix
$S$ and $\epsilon$ satisfying
\eqref{eq:basic_inequality}. Backward error analysis is a cornerstone of numerical linear algebra, and it
gets insight into the computed solutions by quantifying the minimal perturbations to the input data that
make the computed solution an exact one of a certain perturbed \eqref{eq:original problem}.
Particularly, we revisit the standard backward error results, and distinguish and
and shed light on them. We then show
that the solution $x_s$ of the sLS problem \eqref{eq:sketched problem} is that of a
perturbed one of the LS problem \eqref{eq:original problem}, where the coefficient matrix $A$ is perturbed
to $A+E$ and the perturbation matrix $E$ satisfies $\|E\| \leq \epsilon \|A\|$. This result establishes a
key connection between the embedding quality and the backward error, and provides a guarantee on the
stability of $x_s$ as an approximate solution of
the LS problem \eqref{eq:original problem}. This bound not only deepens our theoretical understanding
but also provides a practical guidance for assessing the reliability of sketch-and-solve paradigm. In the
meantime, we consider a number of important quantities, including a further investigation on the size of
$\|r_{ls}-r_{s}\|$ and the solution error $\|x_{ls}-x_{s}\|$. We will
obtain compact upper bounds for them and have new findings.

While solving \eqref{eq:sketched problem} accelerates computation, it also necessitates a careful revisit
of stopping criteria for iterative solvers. We will establish sharp bounds for
$\|A^T r_s\| / (\|A\| \|r_s\|)$
and  $\|r_{ls}-r_s\|/\|r_{ls}\|$ in terms of $\epsilon$. Notice that $\|A^T r_s\| / (\|A\| \|r_s\|)$ is the
(relative) residual norm that the exact solution $x_s$
of \eqref{eq:sketched problem} as an approximate one of \eqref{eq:original problem} in the corresponding normal equation of \eqref{eq:original problem}. They describe
the achievable sizes of
these two quantities when the exact solution $x_s$ of \eqref{eq:sketched problem} is used as an approximate
solution of \eqref{eq:original problem}. Exploiting the fundamental and provable
property that $\|A^T r_s\| / (\|A\| \|r_s\|)$ and  $\|r_{ls}-r_s\|$ are no more than their
counterparts $\|A^T r_k^s\| /
(\|A\| \|r_k^s\|)$ and $\|r_{ls}-r_k^s\|$ where $x_s$ is replaced by any iterate
$x_k^s$ of \eqref{eq:sketched problem} and $r_k^s=Ax_k^s-b$,
we are able to design the most proper and efficient
general-purpose stopping criteria that terminate
iterations once $\|A^T r_k^s\| / (\|A\| \|r_k^s\|)$ or  $\|r_{ls}-r_k^s\|$
becomes almost unchanged for a few consecutive iteration steps $k$, thereby
avoiding unnecessary computation or over-solving
without sacrificing attainable accuracy. Specifically, by taking the convergence
properties of LSQR and LSMR into account, we will prove that it is preferable to
monitor $\|A^T r_k^s\| / (\|A\| \|r_k^s\|)$ in the LSMR algorithm and
$\|r_{ls}-r_k^s\|$ in the LSQR algorithm.
Solving the sLS problem \eqref{eq:sketched problem} using such stopping criteria delivers
a best possible approximate solution of \eqref{eq:original problem}, and further iterations do not improve it any longer. Strikingly, for a fairly small distortion $\epsilon$, say around $10^{-1}$, as we shall see,
$\|A^T r_s\| / (\|A\| \|r_s\|)$
and  $\|r_{ls}-r_s\|/\|r_{ls}\|$ is {\em several}
orders larger than $10^{-8}$,
a very common traditional stopping tolerance. Therefore, with new stopping tolerances, LSQR and LSMR
for \eqref{eq:sketched problem} terminate much more early than they do with traditional stopping criteria
and can thus be performed in the most effective and efficient way.

In summary, our contributions are both theoretical and practical. On the theoretical side, we provide a rigorous analysis on backward errors related to \eqref{eq:sketched problem},
the embedded solution $x_s$
and attainable relative residual norms of $x_s$
as approximate solutions of \eqref{eq:original problem} itself and its normal equation,
bridging the gap of modern randomized algorithms. On both the theoretical and
practical sides, we investigate carefully and propose
general-purpose stopping criteria that achieve the most efficient implementation of
iterative solvers without losing accuracy of computed solutions. We report numerical experiments
to confirm our theoretical results and demonstrate the reliability and efficiency of our new stopping criteria.

The paper is organized as follows. \Cref{sec: pre} reviews necessary preliminaries
and presents our main tasks. Our main theoretical
results are in \cref{sec:main}. In \cref{stopdesign}, we analyze the LSQR and LSMR iterates
and propose two new stopping criteria based on the results in \cref{sec:main}.
We report numerical experiments using LSQR and LSMR in \cref{sec:experiments},
and conclude the paper in \cref{sec:conclusions}.

\section{Preliminaries and our main tasks}
\label{sec: pre}

This section introduces the fundamental concepts and foundational tools necessary for understanding the subsequent analysis on the sLS problem \eqref{eq:sketched problem}. We begin with the problem formulation, followed by an overview of subspace embedding techniques, backward error analysis, and traditional
stopping criteria for iterative solvers.

\subsection{Subspace embedding theory}

A randomized subspace embedding matrix changes \eqref{eq:original problem} to \eqref{eq:sketched problem} by
projecting high-dimensional data into low dimensional ones while preserving geometric structures
with the distortion $\epsilon$. As a consequence, if the same direct solver is used, then we may solve
\eqref{eq:sketched problem} much more efficiently than \eqref{eq:original problem}. This section briefly
reviews three mainstream embedding matrices: SRHT, Gaussian matrices, and sparse embedding matrices.

\subsubsection{Subsampled randomized Hadamard transform (SRHT) matrix}
\label{subsec:srht}

A SRHT matrix $S$ (cf. \cite{doi:10.1137/120874540,doi:10.1137/090771806,meier2022randomizedalgorithmstikhonovregularization,
doi:10.1142/S1793536911000787,doi:10.1137/18M1201068,woodruff2014sketching,WOOLFE2008335}) is a composite of three components:
\begin{itemize}
    \item \textbf{Diagonal matrix} $D \in \mathbb{R}^{m \times m}$: Diagonal entries are $\pm 1$ with equal probability, i.e., Rademacher variables.
    \item \textbf{Walsh-Hadamard transform} $H \in \mathbb{R}^{m \times m}$: An orthogonal transform satisfying $H^TH = mI$ and each element of $H$ is $\pm1$.
    \item \textbf{Subsampling matrix} $P \in \mathbb{R}^{d \times m}$: Uniformly select $d$ rows of $A$ and scale them by $\sqrt{1/d}$.
\end{itemize}
$S$ is defined as
\begin{equation}
    S = P H D \in \mathbb{R}^{d \times m}.
\end{equation}
A probabilistic distortion guarantee is
\begin{equation}\label{eq:probabilistic bound}
    \mathbb{P}\left( \forall x \in \mathbb{R}^m,\ (1-\epsilon)\|x\|^2 \leq \|Sx\|^2 \leq (1+\epsilon)\|x\|^2 \right) \geq 1 - \delta,
\end{equation}
which is ensured, provided the embedding row number satisfies  \cite{doi:10.1137/120874540,woodruff2014sketching}:
\begin{equation}\label{eq: d1}
    d= O\left( \epsilon^{-2} (n+\log m ) \log n \right),
\end{equation}
where $\delta$ is a fixed small constant 
and denotes the failure probability.

When applying a SRHT embedding matrix to $A$,
the computation of the matrix multiplication $SA$ costs $\mathcal{O}(mn \log m)$ flops via the fast
Hadamard transform \cite{meier2022randomizedalgorithmstikhonovregularization}. Using a QR
factorization-based algorithm to solve \eqref{eq:sketched problem} costs additional
$O(\epsilon^{-2}n^3\log n)$ flops.

\subsubsection{Gaussian embedding matrix}
\label{subsec:gauss}

A Gaussian embedding matrix  $S \in \mathbb{R}^{d \times m}$ (cf. \cite{10.1145/276698.276876,doi:10.1137/18M1201068,woodruff2014sketching}) has entries independently sampled from $\mathcal{N}(0, 1/d)$. Its key properties include:
\begin{itemize}
    \item \textbf{Unbiasedness}: $\mathbb{E}[\|Sx\|^2] = \|x\|^2$ for any vector $x$.
    \item \textbf{High-probability isometry}: For any fixed subspace in $\mathbb{R}^m$, the embedded row number $d$ satisfies \cite{woodruff2014sketching}:
\end{itemize}
\begin{equation}\label{eq: d2}
    d= O\left( \epsilon^{-2} n \right).
\end{equation}

When applying a Gaussian embedding matrix to $A \in \mathbb{R}^{m \times n}$, the computation of the matrix multiplication $SA$ costs $O(mnd)$ flops \cite{meier2022randomizedalgorithmstikhonovregularization}.
Using a QR factorization-based algorithm to solve \eqref{eq:sketched problem} costs additional
$O(\epsilon^{-2} n^3 )$ flops.

\subsubsection{Sparse embedding matrix}
\label{subsec:sparse}

A sparse embedding matrix $S \in \mathbb{R}^{d \times m}$ (cf.  \cite{ACHLIOPTAS2003671,towards,10.1145/2488608.2488620,doi:10.1137/1.9781611974331.ch21,
10.1145/2488608.2488621,6686147,10.1007/978-3-662-43948-7_73,doi:10.1137/18M1201068,woodruff2014sketching})  contains only one  nonzero entry ($\pm 1$) per column, with its position chosen uniformly at random. Its key features include:
\begin{itemize}
    \item \textbf{Storage efficiency}: Only $O(m)$ space is required to store nonzero indices and signs.
    \item \textbf{Fast computation}: The computation of the matrix
    multiplication $SA$ costs $O(dn)$ flops.
\end{itemize}

It is known from \cite{doi:10.1137/1.9781611974331.ch21,woodruff2014sketching} that the
 embedding row number $d$ in inequality \eqref{eq:probabilistic bound} is
\begin{equation}\label{eq: d3}
    d= O\left( \epsilon^{-2} n^2  \right).
\end{equation}

When applying a sparse embedding matrix to $A \in \mathbb{R}^{m \times n}$, the computation of the matrix multiplication $SA$ costs $O(nd)$ flops \cite{meier2022randomizedalgorithmstikhonovregularization}.
Using a QR factorization-based algorithm to solve \eqref{eq:sketched problem} costs additional
$O(\epsilon^{-2} n^4 )$ flops.

\subsubsection{Comparison and discussion}
\label{subsec:compare}

SRHT balances computational efficiency and distortion control for large-scale dense datasets by leveraging
fast Fourier-like transforms and near-optimal subspace preservation. Gaussian embeddings, while achieving
minimum number \( O(\epsilon^{-2}n) \) of embedded rows incur \( O(mnd) \)
computational cost of $SA$. It suits for small to moderate dense LS problems. Sparse embeddings,
with their low-memory and constant non-zero entries per column, optimize for streaming environments and
memory-constrained architectures, but they hinge on critical trade-offs: problem
dimensionality, inherent data sparsity patterns, and permissible error-accuracy thresholds for downstream
tasks like LS or low-rank approximation. An important observation from \eqref{eq: d1}--\eqref{eq: d3} is that $\epsilon$ approximately reduces $\sqrt{2}$ times if $d$ is
doubled, a quite limited improvement with the doubled computational cost.

\subsection{Backward error analysis, residual analysis and others}

Backward error analysis is a fundamental tool in numerical linear algebra for assessing the stability of numerical algorithms. Regarding the exact solution $x_s$ of \eqref{eq:sketched problem} as an approximate solution of \eqref{eq:original problem}, we will investigate the minimal backward errors $\Delta A$ and $\Delta b$ such that $x_s$ solves
\[
x_s = \arg\min_x \|(A + \Delta A)x - (b + \Delta b)\|.
\]
The minimal normwise backward error for the LS problem \eqref{eq:original problem} is defined as \cite{doi:10.1137/1.9781611971484,higham2002accuracy}:
\[
\eta_{F}(x_s) = \min \left\{ \|(\Delta A,\theta\Delta b)\|_F \, \Big| \, x_s = \arg\min_x \|(A + \Delta A)x - (b + \Delta b)\| \right\},
\]
where $\|\cdot\|_F$ denotes the Frobenuis norm of a matrix and $0\leq \theta<\infty$ is a parameter
controlling. $\theta\rightarrow \infty$ means that only $A$ is perturbed, and $\theta=0$ means that
only $b$ is perturbed.

In the context of the sketch-and-solve paradigm, we will prove that the embedding process amounts to
taking $\Delta b=0$ and only
introducing an approximately minimal perturbation $\Delta A$ in the LS problem \eqref{eq:original problem}.
In this sense, the sLS problem \eqref{eq:sketched problem} is considered to be a perturbed one of the LS
problem \eqref{eq:original problem}, which will enable us to connect the backward errors,
the errors of solution $x_s$ and residual $r_s$ in terms of
the embedding quality $\epsilon$ and to establish a number of insightful results.

\subsection{Stopping criteria for iterative solvers} \label{tradstop}

For $n$ large, it is preferable to iteratively
solve \eqref{eq:original problem} by the LSQR and LSMR methods.
The LS problem \eqref{eq:original problem} is mathematically equivalent to
its normal equation
\begin{equation}\label{eq:normal equation}
    A^TAx=A^Tb.
\end{equation}
LSQR and LSMR terminate iterations when some prescribed tolerance
is satisfied. For the inconsistent \eqref{eq:original problem}, i.e., $r_{ls}\not=0$,
a general-purpose traditional
stopping criterion is
\begin{equation}\label{trad}
\frac{\|A^T r_k\|}{\|A\| \|r_k\|} \leq tol,
\end{equation}
where \(r_k = Ax_k - b\) denotes the residual vector of the
$k$-th iterate $x_k$ and $tol$ is a given tolerance \cite{10.1145/355984.355989}.
Once the above requirement is met for the first $k$, one has
already solved the problem and terminates iterations. Traditionally, one typically
takes $tol\in [10\epsilon_{\rm mach},\sqrt{\epsilon_{\rm mach}}]$
with $\epsilon_{\rm mach}$ being the machine precision.


Although the sLS problem \eqref{eq:sketched problem} has been regarded as a surrogate of the
LS problem \eqref{eq:original problem} since the publication of \cite{4031351},
their numerical treatments have not yet been considered carefully.
Remind that our ultimate mission is solution of the LS problem \eqref{eq:original problem}, and
solving the sLS problem \eqref{eq:sketched problem} is only an intermediate process of
solving \eqref{eq:original problem}.
We may miss smart computational shortcuts that might exist if we do not properly connect the two
mathematical frameworks.
Therefore, a vital question arises: how to appropriately select the tolerance
$tol$ in \eqref{trad} in order to obtain a best
possible approximate solution of \eqref{eq:original problem} when
iteratively solving \eqref{eq:sketched problem}. An excessively large \(tol\) may
prevent us from obtaining an optimal approximate solution of the LS problem
\eqref{eq:original problem} from iteratively solving
the sLS problem \eqref{eq:sketched problem}, while an overly small \(tol\)
leads to redundant iterations without improving accuracy of approximate solutions, causing unnecessary
computational wastes.

Alternatively, we will also
explore equally important general-purpose stopping criteria for
$\|r_k-r_{ls}\|$ based on its intimate connection to
$\|r_s-r_{ls}\|$, and show how to terminate iterative solvers for
the sLS problem \eqref{eq:sketched problem} at right time.

\section{Theoretical results}\label{sec:main}

\subsection{Residual analysis of the sLS problem}
To examine relationships between $r_s$ and $r_{ls}$, we first reiterate the following
result.

\begin{lemma}[\cite{woodruff2014sketching}]
\label{thm:old}
The sketched solution residual $r_s$ and the original residual $r_{ls}$ satisfy
 \eqref{eq:basic foundation}.
\end{lemma}

We now present the following result, which will be frequently used later.

\begin{theorem} [Geometric preservation]\label{thm:projected_error}
    For $ A \in \mathbb{R}^{m \times n} $ and $ b \in \mathbb{R}^m $,
    let $ S \in \mathbb{R}^{d \times m} $ be a randomized embedding matrix with
    the distortion $ \epsilon \in
    (0, 1) $ on the subspace $ \mathcal{L} = \mathrm{span}(A, b).$
    Then for any vector $ y \in \mathbb{R}^n $, it holds that
  \begin{equation}\label{eq:thm3.2}
      \left\| A^T (S^T S-I)(Ay-b) \right\| \leq \epsilon \|A\| \|Ay-b\|.
  \end{equation}
  \end{theorem}

  \begin{proof}
    By the definition of $ \epsilon $-subspace embedding, for all $ x \in \mathcal{L} $, we have
    $$
    (1 - \epsilon) \|x\|^2 \leq \|Sx\|^2 \leq (1 + \epsilon) \|x\|^2
    $$
    from which it follows that
    $$
    \left| x^T (S^T S - I) x \right| \leq \epsilon \|x\|^2.
    $$
    For any symmetric matrix $ M $, denote by $ \|M|_\mathcal{L}\| $ the spectral norm of $M$ restricted to $\mathcal{L}$. Then
    $$
    \| (S^T S - I)|_\mathcal{L} \| = \sup_{x \in \mathcal{L} \setminus \{0\}} \frac{|x^T (S^T S - I)x|}{\|x\|^2} \leq \epsilon.
    $$
    Thus, for any $ x \in \mathcal{L} $,
    $$
    \| (S^T S-I) x \| \leq \epsilon \|x\|.
    $$
    Replace $x$ by a residual vector $ r = Ay-b\in \mathcal{L}$.
    Then
    $$
    \| (S^T S-I) r \| \leq \epsilon \|r\|.
    $$
    Therefore,
    \begin{align*}
        \|A^T (S^T S-I) r\| \leq \|A^T\|\| (S^T S-I) r \| \leq \epsilon \|A\| \|r\|.
    \end{align*}
  \end{proof}

Exploiting Theorem~\ref{thm:projected_error}, we can establish one of the main results in this paper.

\begin{theorem}
\label{thm:new}
With the notation of \Cref{thm:projected_error}, suppose that $r_{ls}\not=0$. Then
\begin{equation}\label{eq:second_inequality}
\frac{\| r_{ls}-r_s\|}{ \|r_{ls}\|} \leq \sqrt{\frac{2\epsilon}{1 - \epsilon}}.
\end{equation}
\end{theorem}

\begin{proof}
From the normal equation \eqref{eq:normal equation}, we obtain
\begin{align*}
    \|r_{ls}\|^2 &= \|(Ax_{ls}-b)^T (Ax_{ls}-b)\|\\
    &= \|(Ax_{ls}-b)^T b\|\\
    &=\|(Ax_{ls}-b)^T (Ax_{s}-b)\|\\
    &=\|r_{ls}^T r_s\|,
    \end{align*}
where the second and third equalities exploit the property that $A^Tr_{ls}=0$ by noticing
$Ax_{ls}$ and $Ax_s$ are in the $\mathcal{R}(A)$, the column space of $A$.
Therefore,
\begin{align*}
    \|r_{ls}-r_s\|^2&=\|r_s\|^2+\|r_{ls}\|^2-2\|r_s^T r_{ls}\|\\
    &=\|r_s\|^2-\|r_{ls}\|^2.
    \end{align*}
Making use of \eqref{eq:basic foundation}, we have
\begin{align*}
    \frac{\|r_{ls}-r_s\|}{ \|r_{ls}\|}&=\left(\frac{\|r_s\|^2}{\|r_{ls}\|^2}-1\right)^{\frac{1}{2}}\\
    &\leq \sqrt{\frac{2\epsilon}{1-\epsilon}}.
    \end{align*}
\end{proof}

\begin{remark}
    Theorem \ref{thm:new} is more insightful and informative than \Cref{thm:old}.
        Notice that the residual $r_{ls}$ is orthogonal to $\mathcal{R}(A)$.
        Bound \eqref{eq:second_inequality} ensures that $r_s$ aligns closely with the \textit{direction} of
        $r_{ls}$, not just its magnitude, so that $r_s$ adheres to the geometric requirement of orthogonality
        to $\mathcal{R}(A)$.
        In contrast, bound \eqref{eq:basic foundation} ignores directional deviations. In addition,
        bound~\eqref{eq:second_inequality} implies that $r_s$ is a meaningful approximation
        to $r_{ls}$, i.e., its relative error is smaller than one, only when $\epsilon<\frac{1}{3}$.

Bound \eqref{eq:second_inequality} will enable us to analyze the error of $x_s$, backward stability and practical error control simultaneously, as we shall show in the sequel.
\end{remark}

\subsection{Residual analysis of normal equations of the LS and sLS problems}

The heart of randomized subspace embedding techniques
lies in their possible capabilities to preserve geometric fidelity across all potential solutions
simultaneously – a universal guarantee that forms the foundation of our first theorem.
To rigorously analyze properties and capabilities of the sketch-and-solve paradigm, we first establish two
sharp bounds for the residual errors of normal equations of \eqref{eq:original problem} and
\eqref{eq:sketched problem} when considering $x_s$ and $x_{ls}$ to be approximate solutions of
\eqref{eq:original problem} and
\eqref{eq:sketched problem}, respectively.

\begin{theorem}[Relative residual bounds]\label{thm:embedded_stopping}
    With the notation of \Cref{thm:projected_error}, we have
\begin{equation}\label{eq:normal_residual_upper_bound}
\frac{\|A^T r_s\|}{\|A\| \|r_s\|} \leq \epsilon
\end{equation}
and
\begin{equation}\label{eq:corr_normal_residual_upper_bound}
\frac{\|(SA)^T (Sr_{ls})\|}{\|SA\| \|Sr_{ls}\|} \leq \frac{\epsilon}{1-\epsilon}.
\end{equation}
\end{theorem}

\begin{proof}
By the normal equation of \eqref{eq:sketched problem} and the definition of $r_s$, we have
\begin{equation*}
A^T S^T S r_s = 0.
\end{equation*}
Therefore,
$$
\|A^T (I - S^T S)r_s\| = \|A^T r_s\|.
$$
By Theorem~\ref{thm:projected_error}, we have
$$ \|A^T r_s\| \leq \epsilon \|A\| \|r_s\|,$$
which proves \eqref{eq:normal_residual_upper_bound}.

Exploiting $A^Tr_{ls}=0$, we obtain
\begin{equation*}
    (SA)^T (Sr_{ls})=A^T (S^T S-I)r_{ls}.
\end{equation*}
From Theorem~\ref{thm:projected_error}, we obtain
\begin{equation}\label{eq:temp2}
    \|(SA)^T (Sr_{ls})\|\leq\epsilon\|A\|\|r_{ls}\|.
\end{equation}
From \eqref{eq:basic_inequality}, we have
\begin{equation}\label{lowerb}
    \|SA\|^2\geq(1-\epsilon)\|A\|^2,\quad \|Sr_{ls}\|^2\geq(1-\epsilon)\|r_{ls}\|^2.
\end{equation}
Substituting them into the right-hand side of \eqref{eq:temp2} yields
\begin{equation*}
    \|(SA)^T (Sr_{ls})\|\leq\frac{\epsilon}{1-\epsilon}\|SA\|\|Sr_{ls}\|,
\end{equation*}
which establishes \eqref{eq:corr_normal_residual_upper_bound}.
\end{proof}

\begin{remark}
This theorem quantifies how embedding matrices
control the alignments between the residual $r_s$ of the sketched solution $x_s$
and the column space $\mathcal{R}(A)$ of $A$ as well as between $Sr_{ls}$
and the sketched column space $\mathcal{R}(SA)$.
Ratio \eqref{eq:normal_residual_upper_bound} approximately measures the cosine of the angle between $r_s$ and $\mathcal{R}(A)$. For the sLS problem \eqref{eq:sketched problem}, ratio
\eqref{eq:corr_normal_residual_upper_bound} analogously measures the cosine of the angle between $Sr_{ls}$ and $\mathcal{R}(SA)$, where we consider the solution $x_{ls}$ of the original LS problem \eqref{eq:original problem} as an approximate solution of the sLS problem \eqref{eq:sketched problem}.
Bound \eqref{eq:normal_residual_upper_bound} ensures near-orthogonality between $r_s$ and $\mathcal{R}(A)$ within an $\epsilon$-tolerance, while bound \eqref{eq:corr_normal_residual_upper_bound} ensures near-orthogonality between $Sr_{ls}$ and $\mathcal{R}(SA)$ within a $\frac{\epsilon}{1-\epsilon}$-tolerance,
a most conservative estimate, as is seen from \eqref{lowerb}. The denominator $1-\epsilon$ in \eqref{eq:corr_normal_residual_upper_bound} is, on average, one.
\end{remark}

\subsection{Backward error analysis of the sLS problem}
Our analysis is based on the characterization of backward stability of LS
solutions \cite[ p.~392, Theorem 20.5]{higham2002accuracy}.

\begin{lemma}[\cite{higham2002accuracy}]\label{thm:higham}
    Let $\bar{x}$ be an approximate solution of the LS problem \eqref{eq:original problem}. The normwise backward error
    \begin{equation}
    \eta_F(\bar{x})=\min\{\|(E,\theta e)\|_F \ |\  \|(A+E)x-(b+e)\|=\min\}
    \end{equation}
    is given by
    \begin{equation}\label{2}
    \eta_F(\bar{x})=\left\{\begin{array}{lr}
    \gamma\sqrt{\mu},& \lambda_*\geq 0,\\
    (\gamma^2\mu+\lambda_*)^{1/2}, & \lambda_*<0,
    \end{array}
    \right.
    \end{equation}
    where $\bar{r}=A\bar{x}-b, \ \gamma=\|\bar{r}\|/\|\bar{x}\|$, and
    \begin{equation}
    \lambda_*=\lambda_{\min}\left(AA^T-\mu\frac{\bar{r}\bar{r}^T}{\|\bar{x}\|^2}\right),\ \
    \mu=\frac{\theta^2\|\bar{x}\|^2}{1+\theta^2\|\bar{x}\|^2}<1.
    \end{equation}
    Particularly, $\theta\rightarrow\infty$ means that only $A$ is perturbed and $b$ remains unchanged, so that $\mu=1$.
    \end{lemma}

Below we investigate this lemma by distinguishing $r_{ls}\not=0$ and $r_{ls}=0$, which
correspond to the inconsistent and
consistent \eqref{eq:original problem}, respectively,
and prove what values $\eta_F(\bar{x})$ takes.

\begin{theorem}\label{thm:teacher}
With the notation of \cref{thm:higham}, suppose that $A$ is of full column rank and $r_{ls}\not=0$.
Then $\lambda_*<0$, i.e.,
\begin{equation*}
\eta_F(\bar{x})= (\gamma^2\mu+\lambda_*)^{1/2},
\end{equation*}
and
\begin{equation}\label{5}
\eta_F(\bar{x})\leq \frac{\|A^T\bar{r}\|}{\|\bar{r}\|}.
\end{equation}
If $r_{ls}=0$ and $\bar{x}$ is sufficiently
close to $x_{ls}$, then $\lambda_*\geq 0$, i.e.,
    \begin{equation}\label{eq: extra_back}
        \eta_F(\bar{x})= \gamma\sqrt{\mu}.
    \end{equation}
\end{theorem}

\begin{proof}
        For $r_{ls}\not=0$ and $m>n$,  there exists $y$ $\in \mathbb{R}^m$ with $\|y\|=1$ such that
        \begin{equation*}
            A^T y=0, \quad \bar{r}^T y=(A\bar{x}-b)^T y =-b^T y\neq 0.
        \end{equation*}
    Based on the basic property that the smallest eigenvalue of a real symmetric matrix
    is its smallest Rayleigh quotient, we obtain
        \begin{equation*}
            \lambda_*\leq y^T\left(AA^T-\mu\frac{\bar{r}\bar{r}^T}{\|\bar{x}\|^2}\right)y
            =-\mu\frac{(\bar{r}^T y)^2}{\|\bar{x}\|^2}<0.
        \end{equation*}
    Since
    \begin{eqnarray*}
    \lambda_* &\leq & \frac{\bar{r}^T\left(AA^T-\mu\frac{\bar{r}\bar{r}^T}{\|\bar{x}\|^2}\right)\bar{r}}{\|\bar{r}\|^2}\\
    &=& \frac{\|A^T\bar{r}\|^2}{\|\bar{r}\|^2}-\mu\frac{\|\bar{r}\|^2}{\|\bar{x}\|^2}\\
    &=& \frac{\|A^T\bar{r}\|^2}{\|\bar{r}\|^2}-\gamma^2\mu,
    \end{eqnarray*}
    we obtain
    \begin{equation*}
    \eta_F(\bar{x})=(\gamma^2\mu+\lambda_*)^{1/2}\leq \frac{\|A^T\bar{r}\|}{\|\bar{r}\|},
    \end{equation*}
   proving \eqref{5}.

For $r_{ls}=0$, we have $ b=Ax_{ls} $. Therefore,
    \begin{equation*}
        AA^T-\mu\frac{\bar{r}\bar{r}^T}{\|\bar{x}\|^2}=
        A\left(I-\mu\frac{(\bar{x}-x_{ls})(\bar{x}-x_{ls})^T}{\|\bar{x}\|^2}\right)A^T.
    \end{equation*}
    For $\bar{x}$ sufficiently close to $x_{ls}$, since $\mu<1$, the matrix $I-\mu\frac{(\bar{x}-x_{ls})(\bar{x}-x_{ls})^T}{\|\bar{x}\|^2}$ must be
    symmetric positive semi-definite, so that
     \begin{equation*}
       A\left(I-\mu\frac{(\bar{x}-x_{ls})(\bar{x}-x_{ls})^T}{\|\bar{x}\|^2}\right)A^T
    \end{equation*}
    is positive semi-definite too. Thus, we have $\lambda_*\geq 0$, i.e., \eqref{eq: extra_back} holds.
    \end{proof}

\begin{remark}\label{rem9}
    By $A^Tr_{ls}=0$,
    it is straightforward to verify that $r_{ls}/\|r_{ls}\|$ is a
    unit-length eigenvector of $AA^T-\mu\frac{r_{ls}r_{ls}^T}{\|x_{ls}\|^2}$ corresponding
    to its smallest eigenvalue $-\mu\|r_{ls}\|^2/\|x_{ls}\|^2<0$.
   For $\bar{x}$ sufficiently close to $ x_{ls}$,  the residual
   $\bar{r}$ must be sufficiently close to $r_{ls}$ too.
   As a result, by continuity arguments, $\lambda_{\min}\left(AA^T-\mu\frac{\bar{r}\bar{r}^T}{\|\bar{x}\|^2}\right)$ must be negative.
        Bound \eqref{5} shows that $\eta_F(\bar{x})\rightarrow 0$ as $\bar{x}\rightarrow
    x_{ls}$. Moreover, this bound is also sharp since $\bar{r}\rightarrow r_{ls}$ and $\bar{r}/\|\bar{r}\|$ is naturally an excellent approximation
    to the eigenvector of $AA^T-\mu \frac{\bar{r}\bar{r}^T}{\|\bar{x}\|^2}$ corresponding to its smallest eigenvalue $\lambda_*$.
As a matter of fact, we find out that bound \eqref{5} is the same order small as the lower bound for $\eta_F(\bar{x}$) in \cite[p.~279]{https://doi.org/10.1002/nla.1680020308} before Theorem 3.2 in it
    and both tend to zero as fast as $\|A^T\bar{r}\|$. These two arguments together
    show that bound \eqref{5} is indeed sharp. Therefore, we
   can use $\|A^T\bar{r}\|/\|\bar{r}\|$ to accurately estimate the minimal backward error $\eta_F(\bar x)$
   for $r_{ls}\not=0$ once $\bar{r}$ is close to $r_{ls}$.
    \end{remark}

    \begin{remark}
        For $r_{ls}=0$, the condition that $\bar{x}$ is sufficiently close to $x_{ls}$ is necessary for
        \eqref{eq: extra_back} to hold. Otherwise, it is likely that $\lambda_*<0$,
        so that  \eqref{eq: extra_back} fails to hold. Here is an example:
        Take an $\bar{x}$ satisfying $\|\bar{x}\|^2<\frac{\mu}{1+\mu}\|x_{ls}-\bar{x}\|^2$, and let
        $A^T y=x_{ls}-\bar{x}$ for some $y$. Then
        \begin{eqnarray*}
        \lambda_*&\leq& \frac{y^TA\left(I-\mu\frac{(\bar{x}-x_{ls})(\bar{x}-x_{ls})^T}{\|\bar{x}\|^2}\right)A^T y}{\|y\|^2}\\
        &=&\frac{\|\bar{x}\|^2\|\bar{x}-x_{ls}\|^2-\mu\|\bar{x}-x_{ls}\|^4}{\|\bar{x}\|^2\|y\|^2}\\
        &\leq& \frac{\mu\|\bar{x}-x_{ls}\|^4-\mu(1+\mu)\|\bar{x}-x_{ls}\|^4}{(1+\mu)\|\bar{x}\|^2\|y\|^2}\\
        &=&\frac{-\mu^2\|\bar{x}-x_{ls}\|^4}{(1+\mu)\|\bar{x}\|^2\|y\|^2}<0.
        \end{eqnarray*}
    \end{remark}

    \begin{remark}
        Equality \eqref{eq: extra_back} shows that $\eta_F(\bar{x})\rightarrow 0$ as $\bar{x}\rightarrow x_{ls}$, and it is included for completeness. Throughout this paper, we focus on the more general
        inconsistent case $r_{ls}\not=0$.
    \end{remark}

    To get insight into the size of minimal backward error of $x_s$ as an approximate solution of \eqref{eq:original problem}, we need the following result.

    \begin{lemma}[\mbox{\cite[p.~34]{doi:10.1137/1.9781611971484}}]
    For $r_{ls}\not=0$, the perturbation $E=-\bar{r}\bar{r}^{\dagger}A$ satisfies
    $$
    \bar{x}=\arg\min \|(A+E)x-b\|,
    $$
    and
    \begin{equation}\label{6}
    \|E\|=\frac{\|A^T\bar{r}\|}{\|\bar{r}\|},
    \end{equation}
    where $\dagger$ denotes the Moore--Penrose generalized inverse of a matrix.
    \end{lemma}

   By \cref{rem9}, this lemma indicates that $\|E\|=\|\bar{r}\bar{r}^\dagger A\|$ is
   a sharp bound for the minimal backward error $\eta_F(\bar{x})$ when $\bar{r}$ is close to $r_{ls}$.

\begin{theorem}[Explicit backward perturbations]\label{thm:explicit_error}
For the sLS problem \eqref{eq:sketched problem}, its solution $x_s$
solves the perturbed \eqref{eq:original problem} with two explicit backward perturbation matrices
\begin{equation}\label{twofor}
E_1 = -r_sr_s^\dagger A,\quad E_2=(r_{ls}-r_s) x_s^{\dagger}.
\end{equation}
Furthermore,
\begin{equation}\label{eq:E1}
    \|E_1\| \leq \|A\|\epsilon,
\end{equation}
and
\begin{equation}\label{eq:E2}
    \|E_2\| \leq \frac{\|r_{ls}\|}{\|x_s\|} \sqrt{\frac{2\epsilon}{1 - \epsilon}}.
\end{equation}
\end{theorem}

\begin{proof}
For the construction of the matrices in \eqref{twofor} and the proof, we refer to
\cite[p.~98]{doi:10.1137/1.9781611971484} and the reference therein when considering
$x_s=\arg\min \|(A+E)x-b\|$.
We only need to prove \eqref{eq:E1} and \eqref{eq:E2}. Using the spectral norm inequality and \eqref{eq:normal_residual_upper_bound}, we have
    \begin{equation*}
    \|E_1\| = \|r_sr_s^\dagger A\| = \frac{\|r_sr_s^T A\|}{\|r_s\|^2} =\frac{\|A^Tr_s\|}{\|r_s\|}\leq \|A\|\epsilon.
    \end{equation*}
   By Theorem~\ref{thm:new}, we obtain
      \begin{equation*}
    \|E_2\| = \frac{\|r_{ls}-r_s\|}{\|x_s\|}\leq  \frac{\|r_{ls}\|}{\|x_s\|}\sqrt{\frac{2\epsilon}{1 - \epsilon}},
    \end{equation*}
which establishes \eqref{eq:E2}.
\end{proof}
\begin{remark}
Comparative analysis with the preceding minimal backward error estimates reveals that the upper bound in \eqref{eq:E1} exactly coincides with that of the minimal backward error. That is, the sharpness of \eqref{eq:E1} demonstrates that $\|E_1\|$ constitutes an approximately minimal backward error bound
as $\epsilon\rightarrow 0$. For the
inconsistent LS problem \eqref{eq:original problem}, bound \eqref{eq:E2} is inferior to bound
\eqref{eq:E1} and
is only $O(\sqrt{\epsilon})$ as $\epsilon\rightarrow 0$.
\end{remark}

Next we connect backward and forward error analysis, and establish error bounds for $x_s$ in terms of
$\epsilon$ and $r_s$ or $r_{ls}$.

\begin{theorem}[Approximate solution error bounds]\label{thm:error_bounds}
Let $\kappa(A)=\|A\|\|A^\dagger\|$. Then
\begin{align}
\frac{\|x_{ls} - x_s\|}{\|x_s\|} &\leq \kappa^2(A)\epsilon\frac{\|r_s\|}{\|A\|\|x_s\|}, \label{eq:bound1} \\
\frac{\|x_{ls} - x_s\|}{\|x_{ls}\|} &\leq \kappa^2(A)\epsilon\sqrt{\frac{1+\epsilon}{1-\epsilon}}
\frac{\|r_{ls}\|}{\|A\|\|x_{ls}\|}.  \label{eq:bound1b}
\end{align}
\end{theorem}

\begin{proof}
Exploiting $A^T r_{ls}=0$, from $r_{ls}=Ax_{ls}-b$ and $r_s=A x_s-b$, we obtain
\begin{equation}
A^TA(x_{ls}- x_s)=A^T (r_{ls}-r_s)=-A^T r_s.
\end{equation}
    Taking norms on both sides gives
\begin{equation*}
        \sigma_{\min}^2(A)\|x_{ls} - x_s\| \leq \|A^T r_s\|,
\end{equation*}
where $\sigma_{\min}(\cdot)$ is the smallest singular value of a matrix.
By  \eqref{eq:normal_residual_upper_bound}, we obtain
$$
  \sigma_{\min}^2(A)\|x_{ls} - x_s\| \leq \epsilon \|A\| \|r_s\|.
$$
Therefore,
$$
    \frac{\|x_{ls} - x_s\|}{\|x_s\|} \leq \kappa^2(A)\epsilon\frac{\|r_s\|}{\|A\|\|x_s\|},
$$
which proves \eqref{eq:bound1}.

Exploiting \eqref{eq:basic foundation}, we amplify bound \eqref{eq:bound1} and can rewrite the resulting bound as
\eqref{eq:bound1b}.
\end{proof}

\begin{remark}
Bounds \eqref{eq:bound1} and \eqref{eq:bound1b} reveal a geometric stability inherent to $\epsilon$ and the
sizes of $\|r_s\|$ and $\|r_{ls}\|$, respectively. They demonstrate how embedding-induced distortions propagate
and affect the accuracy of $x_s$, and enable us to quantify solution quality.
\end{remark}

\begin{remark}
If $r_{ls}=0$, then $r_s=0$, provided that the matrix \( S \) ensures that \( SA \) maintains full column
rank, a condition under which \eqref{eq:sketched problem} is equivalent to \eqref{eq:original problem}.
This equivalence guarantees that the solution \( x_s =x_{ls}\),
independently of the parameter \( \epsilon \) and the conditioning of $A$. In this case, bounds
\eqref{eq:bound1} and \eqref{eq:bound1b} are trivially zero. Moreover, if the
inconsistence degree of \eqref{eq:original problem} is weak, i.e., $\|r_{ls}\|/(\|A\|\|x_{ls}\|)$ is  small,
then $x_s$ can be a good approximation to $x_{ls}$ even though $\epsilon$ is only fairly small; otherwise,
$x_s$ may be a poor approximation to $x_{ls}$ and even may have no accuracy when
$\kappa(A)$ is large and $\epsilon$ is not small enough. The stronger the inconsistence degree is,
the smaller $\epsilon$ is required to obtain equally accurate approximate solutions $x_s$'s of
\eqref{eq:original problem}. More precisely, this theorem indicates that, for the same $A$ but different
right-hand sides, equally accurate approximate
solutions require that the products $\epsilon\|r_{ls}\|$'s be a constant, which is different from
\eqref{eq:second_inequality} where the relative residual
error $\|r_{ls}-r_s\|/\|r_{ls}\|$ only depends on the size of $\epsilon$.
\end{remark}

With the help of \Cref{thm:error_bounds}, we can revisit
the left-hand side of \eqref{eq:second_inequality} in \Cref{thm:new} and obtain a new bound.

\begin{theorem}
With the notation of \Cref{thm:projected_error}, we have
\begin{equation}\label{eq: second_new_bound}
\frac{\| r_{ls}-r_s\|}{ \|r_{ls}\|} \leq \min\left\{\sqrt{\frac{2\epsilon}{1 - \epsilon}},\kappa^2(A)\epsilon\sqrt{\frac{1+\epsilon}{1-\epsilon}}\right\}.
\end{equation}
\end{theorem}

\begin{proof}
    Since
    \begin{equation}\label{9}
        \|r_{ls}-r_s\|=\|A(x_{ls}-x_s)\|\leq \|A\|\|x_{ls}-x_s\|,
    \end{equation}
    by \eqref{9} and \eqref{eq:bound1}, we obtain
    \begin{align*}
        \frac{\|r_{ls}-r_s\|}{\|r_{ls}\|}&\leq\frac{\|A\|\|x_{ls}-x_s\|}{\|r_{ls}\|}\\
        &\leq \kappa^2(A)\epsilon\frac{\|r_s\|}{\|r_{ls}\|},
    \end{align*}
from which and \eqref{eq:basic foundation} it follows that
\begin{equation}\label{10}
\frac{\| r_{ls}-r_s\|}{ \|r_{ls}\|} \leq \kappa^2(A)\epsilon\left(\frac{1+\epsilon}{1-\epsilon}\right)^{\frac{1}{2}}.
\end{equation}
Together with \eqref{eq:second_inequality}, we complete the proof.
\end{proof}

\begin{remark}
    Bound \eqref{eq: second_new_bound} reveals how the condition number $\kappa(A)$ affects the error
    $\|r_{ls}-r_s\|$. If $\kappa(A)< \sqrt[4]{\frac{2}{(1+\epsilon)\epsilon}}< \frac{1.2}{\epsilon^{1/4}}$,
    bound \eqref{eq: second_new_bound} takes
    the second term in the braces; otherwise, it takes the first term, i.e., \eqref{eq:second_inequality},
    which should be a general case since, in applications,
    one generally cannot expect that $A$ is so well conditioned that $\kappa(A)<\frac{1.2}{\epsilon^{1/4}}$
    for a given fairly small $\epsilon$. For instance, for
    $\epsilon=10^{-4}$, it requires $\kappa(A)<12$, meaning that $A$ is very well conditioned.
\end{remark}

\subsubsection{Acute angle perturbations and embeddings}
The previous backward error analysis motivates us to study more on perturbations
that satisfy geometric constraints. Let us start with a definition.

\begin{definition}[Acute perturbation \mbox{\cite[pp.~137-146]{stewart1990matrix}}]\label{def:acute_p}
Let $\tilde{A}=A+E$. The subspaces $\mathcal{R}(A)$ and $\mathcal{R}(\tilde{A})$ are said to be acute if the corresponding orthogonal projectors $P_{\mathcal{R}(A)}$ and $P_{\mathcal{R}(\tilde{A})}$ satisfy
$$\|P_{\mathcal{R}(A)}-P_{\mathcal{R}(\tilde{A})}\|<1.$$
Correspondingly, $\tilde{A}=A+E$ is said to be an acute perturbation of A if $\mathcal{R}(A)$ and $\mathcal{R}(\tilde{A})$ as well as $\mathcal{R}(A^T)$ and $\mathcal{R}(\tilde{A}^T)$ are acute.
\end{definition}

\begin{lemma}[\mbox{\cite[p.~27]{doi:10.1137/1.9781611971484}}]\label{lemma:eq_relation_acute}
    The matrix $\tilde{A}$ is an acute perturbation matrix of A if and only if
    $$\mathrm{rank}(A)=\mathrm{rank}(\tilde{A})=\mathrm{rank}(P_{\mathcal{R}(A)}\tilde{A}P_{\mathcal{R}(A^T)}).$$
\end{lemma}
Based on this lemma, we now introduce a special class of subspace embeddings that guarantee to induce acute angle perturbations.

\begin{definition}[Acute angle embedding]\label{def:acute_embed}
An embedding matrix $S \in \mathbb{R}^{s \times m}$ is called an acute angle embedding of $A\in \mathbb{R}^{m \times n}$  with ${\rm rank}(A)=n$ if there exists an acute perturbation $\tilde{A}=A+E$ of $A$ with the minimal perturbation $E$ such that the unique
solution $x_s$ of \eqref{eq:sketched problem} satisfies
\begin{equation}
    x_s = \arg\min_y \|(A+E)y - b\|.
\end{equation}
\end{definition}

Based on the above, we investigate what types of embeddings can preserve acute angles. Let us review
two fundamental results for non-acute and acute perturbations, though they are not used in our later proofs.

\begin{lemma}[\mbox{\cite[p.~140]{stewart1990matrix}}]\label{lem:acute_perturb}
    If $A$ and $A+E$ is not acute, then
    \begin{equation}
        \|A^\dagger-\tilde{A}^\dagger\|\geq \frac{1}{\|E\|}.
    \end{equation}
\end{lemma}

This lemma implies that if ${\rm rank}(\tilde{A})<n$ then the minimal
2-norm embedded solution $x_s$ can be far from the solution $x_{ls}$ of the LS problem
\eqref{eq:original problem} and generally bears no relation to $x_{ls}$.

\begin{lemma}[\mbox{\cite[p.~145]{stewart1990matrix}}]\label{lem:wedin}
    Let $A\in \mathbb{R}^{m \times n}$ with $m\geq n$ and $\tilde{A}=A+E$. If $\mathrm{rank}(A)=\mathrm{rank}(\tilde{A})=n$, then
    \begin{equation*}
        \|\tilde{A}^\dagger-A^\dagger\|   \leq  \sqrt{2} \|\tilde{A}^\dagger\| \|A^\dagger\|\|E\|.
    \end{equation*}
\end{lemma}

This lemma shows that an acute perturbation $\tilde{A}^\dagger$ is continuous and tends to
$A^\dagger$ as $E$ tends to zero, so that the unique $x_s\rightarrow x_{ls}$ by the standard perturbation
theory in \cite{stewart1990matrix}.

\begin{theorem}[Acute angle embedding criterion]\label{thm:min_sv_bound}
Let $S \in \mathbb{R}^{d \times m}$ be an $\epsilon$-distortion embedding matrix over $\mathcal{R}((A,b))$,
and assume that ${\rm rank}(A)=n$. Then provided $\kappa(A)\epsilon< 1$,
the matrix $S$ is an acute angle embedding.
\end{theorem}

\begin{proof}


According to Lemma~\ref{lemma:eq_relation_acute}, an acute angle embedding requires that
        \begin{equation*}
            \mathrm{rank}(A)=\mathrm{rank}(A+E),
        \end{equation*}
        which holds when
        $$
        \|E\|<\sigma_{\min}(A).
        $$
By \eqref{eq:E1}, this requirement is met once
        \begin{align*}
            \epsilon\kappa(A) < 1.
        \end{align*}
\end{proof}

\begin{remark}
This theorem gives a sufficient condition for ${\rm rank}(SA)={\rm rank}(A+E)=n$.
However, $\kappa(A)\epsilon< 1$ may be restrictive. Fortunately, it is not {\em necessary},
so that it is not a stringent requirement on $\epsilon$ in practice.
\end{remark}

\section{Stopping criteria of iterative solvers for the sLS problem}\label{stopdesign}

Unlike those traditional stopping criteria for an iterative solver mentioned in
\cref{tradstop}, the sharp bounds \eqref{eq:normal_residual_upper_bound}
and \eqref{eq: second_new_bound} inspire us to
adopt completely new and general-purpose stopping criteria when iteratively solving the sLS
problem \eqref{eq:sketched problem} and using iterates as approximate solutions of of the
LS problem \eqref{eq:original problem}. For this
purpose, based on \eqref{eq:normal_residual_upper_bound}, a
stopping criterion for iterative solvers requires that one terminates iterations at the smallest iteration $k$
whenever the current iterate $x_k^s$ of \eqref{eq:sketched problem} satisfies
    \begin{equation}\label{eq:normal_residual_upper_bound2}
        \frac{\|A^T r_k^s\|}{\|A\| \|r_k^s\|} \leq \epsilon,
    \end{equation}
    where $r_k^s=Ax_k^s-b$.
This fundamentally different criterion from the traditional \eqref{trad}
arises from two important facts. First, the sharpness of \eqref{eq:normal_residual_upper_bound}
implies that the inherent error introduced by randomization
already dominates when the inequality becomes equality. Consequently, further iterations on the
sLS problem \eqref{eq:sketched problem} can only reduce errors
specific to itself but cannot bring out improvements for the original LS problem \eqref{eq:original problem}.
That is, once an iterate $x_k^s$ reached the bound $\epsilon$ in \eqref{eq:normal_residual_upper_bound2}, continuing iterations would no longer improve
the accuracy of $x_k^s$ as approximate solutions of the LS problem \eqref{eq:original problem}.
Second, in terms of \eqref{eq:normal_residual_upper_bound}, the left-hand side of \eqref{eq:normal_residual_upper_bound2} quantifies the orthogonality of
$r_k^s$ against $\mathcal{R}(A)$. This shows that the randomized approximation
framework fundamentally limits the attainable accuracy of $x_k^s$ as approximate solutions of
\eqref{eq:original problem} and surpassing this threshold wastes computational effort.

As far as an iterative solver for \eqref{eq:sketched problem} itself is concerned, a traditional
stopping criterion like \eqref{trad} is to terminate the solver for the first iteration step $k$ such that
\begin{equation}\label{eq:normal_residual_upper_bound3}
        \frac{\|(SA)^T (Sr_k^s)\|}{\|SA\| \|Sr_k^s\|} \leq tol,
\end{equation}
where one typically takes $tol\in [10\epsilon_{\rm mach},\sqrt{\epsilon_{\rm mach}}]$. However,
as has been elaborated above, such traditional criterion is inappropriate for
any longer because it ignores a key issue of the sLS problem \eqref{eq:sketched problem}'s
mismatch with the original LS problem \eqref{eq:original problem}. A reasonable and
correct stopping criterion is to use \eqref{eq:normal_residual_upper_bound2} rather than
\eqref{eq:normal_residual_upper_bound3} when iteratively
solving \eqref{eq:sketched problem}, which
enforces an optimal residual alignment guaranteed by the randomized framework.

We now further show why \eqref{eq:normal_residual_upper_bound2} is valid and
shed more light on it. To do this, let us write
$$
\|(SA)^T(Sr_k^s)\|=\|A^TS^TS r_k^s\|=\| A^T(S^TS-I) r_k^s+A^T r_k^s\|.
$$
Therefore, exploiting \eqref{eq:thm3.2}, we obtain
\begin{equation}\label{basicrelation}
\left|\|(SA)^T Sr_k^s\|- \|A^T(S^TS-I) r_k^s\|\right|\leq \|A^Tr_k^s\|\leq \|(SA)^T Sr_k^s\|+\|A\|\|r_k^s\|\epsilon.
\end{equation}
For a convergent iterative solver,
since $ A^Tr_k^s\rightarrow A^Tr_s$ and $(SA)^T Sr_k^s\rightarrow 0$ with $k$
increasing, we ultimately have
\begin{equation}\label{lowerupper}
\|A^T(S^TS-I) r_k^s\| \lesssim \|A^Tr_k^s\|\lesssim \|A\|\|r_k^s\|\epsilon
\end{equation}
i.e., \eqref{eq:normal_residual_upper_bound2} ultimately holds.

\begin{remark}
We notice from the proof of \cref{thm:projected_error}
that bound \eqref{eq:thm3.2} is sharp, so that the lower and upper bounds in
\eqref{lowerupper} are
almost equal and thus \eqref{eq:normal_residual_upper_bound2} is sharp.
As a result, once $\|(SA)^T Sr_k^s\|$ is sufficiently small, $\frac{\|A^Tr_k^s\|}{\|A\|\|r_k^s\|}$ stablizes at no more than $\epsilon$. Moreover, once $\|(SA)^T Sr_k^s\|<\|A\|\|r_k^s\|\epsilon$ considerably, say
 $\|(SA)^T Sr_k^s\|\leq 0.1\cdot\|A\|\|r_k^s\|\epsilon$,
further iterations do not improve the accuracy of $x_k^s$'s as approximate solutions
of the LS problem \eqref{eq:original problem}, and $\|A^Tr_k^s\|$ starts to stabilize with $k$ increasing.
Coming back to \eqref{eq:normal_residual_upper_bound3} and noticing that
$\|SA\| \|Sr_k^s\|\approx \|A\|\|r_k^s\|$, we can deduce that it suffices to take
$tol=0.1\epsilon$.
\end{remark}

Remarkably, the distortion $\epsilon$ or its accurate estimate are generally not available in
computations. Nevertheless, motivated by
\eqref{eq:normal_residual_upper_bound2}, we can design such a stopping criterion: If the left-hand sides of
\eqref{eq:normal_residual_upper_bound2} are almost unchanged for a few consecutive iteration steps $\ell$, e.g., $5$, from some iteration step $k$ onwards, we terminate the solver.
Precisely, for some $k$ onwards, if the geometric mean of the $k$th to $(k+\ell)$-th
relative residual
norms in the left-hand sides of \eqref{eq:normal_residual_upper_bound2}
are nearly one, we terminate the solver, and claim to have found a  best
possible approximate solution of \eqref{eq:original problem}. Specifically,
we have obtained an optimal approximate solution $x_k^s$ of the LS problem
\eqref{eq:original problem} whenever
\begin{equation}\label{stopcrit1}
\left(\frac{\|A^T x_{k+\ell}^s\|\|x_k^s\|}{\|A^T x_k^s\|\|x_{k+\ell}^s\|}\right)^{1/\ell}\approx 1,
\end{equation}
e.g., $0.99\sim 1.01$, for the first iteration step $k$. This is our first stopping criterion.

Alternatively, based on \eqref{eq: second_new_bound}, we can propose the other stopping criterion.
By \eqref{eq:basic_inequality} we have
\begin{equation}\label{eq: basic_inequality_c}
        \frac{1}{1+\epsilon}\|Sr_k^s\|^2\leq \|r_k^s\|^2 \leq \frac{1}{1-\epsilon}\|Sr_k^s\|^2.
\end{equation}
Since $Sr_k^s\rightarrow Sr_s$ and $r_s$ satisfies
\eqref{eq: second_new_bound}, we have $r_k^s\rightarrow r_s$. Therefore, if $\|r_k^s\|$ starts to
stabilize for a few consecutive steps $\ell$, say 5, from some $k$ onwards,
we terminate the iterations and have solved the LS problem \eqref{eq:original problem}.
Specifically, an iterative solver has computed
an optimal approximate solution of \eqref{eq:original problem} whenever the geometric mean
\begin{equation}\label{stopcrit2}
\left(\frac{\|r_{k+\ell}^s\|}{\|r_k^s\|}\right)^{1/\ell}\approx 1,
\end{equation}
e.g., $0.99\sim 1.01$,
for the first iteration step $k$. We then terminate the solver when \eqref{stopcrit2} is met.
This is our second stopping criterion.

For the LSQR algorithm, it is well known from \cite{doi:10.1137/1.9781611971484,10.1145/355984.355989} that
the left-hand side of \eqref{eq:normal_residual_upper_bound3} may not decrease smoothly
as $k$ increases and may thus be hard to stabilize smoothly when $SA$ is ill conditioned, so does
the left-hand
side of \eqref{eq:normal_residual_upper_bound2} either; on the other hand,
$\|Sr_k^s\|$ unconditionally decreases monotonically and converges
to $\|Sr_s\|$, so that $\|r_k^s\|\rightarrow \|r_s\|$ exhibits smooth convergence for a fairly small
$\epsilon$, as \eqref{eq: basic_inequality_c} indicates.
Therefore, \eqref{eq:basic foundation} shows that $\|r_k^s\|$ ultimately stabilizes
at $\|r_{ls}\|$ with a multiple $\sqrt{\frac{1+\epsilon}{1-\epsilon}}$.
As a consequence, the termination criterion \eqref{stopcrit1} for LSQR may be unreliable, but
 \eqref{stopcrit2} is a reliable and general-purpose one for it.

The LSMR algorithm is mathematically the minimal
residual (MINRES) method \cite{doi:10.1137/1.9781611971484} applied to the normal equation of a LS problem. In our context, this means
that, when using it to solve \eqref{eq:sketched problem}, $\|(SA)^T(Sr_k^s)\|$ rather than $\|Sr_k^s\|$ decreases monotonically with $k$ increasing.
According to \eqref{basicrelation} and the analysis followed, it is known
that $\|A^Tr_k^s\|$ overall decreases smoothly but
$\|r_k^s\|$ may not since $\|Sr_k^s\|$ does not possess monotonic decreasing property and
may exhibit irregular convergence. Therefore, if LSMR is used to solve \eqref{eq:sketched problem},
then the stopping criterion \eqref{stopcrit1} is preferable to \eqref{stopcrit2}.


\section{Numerical experiments}\label{sec:experiments}
We now present a number of numerical experiments to justify our theoretical results and new
stopping criteria.
The experiments were performed on Intel(R) Core(TM) i7-9700 CPU with 7.8GB RAM using the {\sc Matlab}
R2024a with
the machine precision $\epsilon_{\rm mach}=2.22\times 10^{-16}$ under Microsoft Windows 11 system, and we
used the {\sc Matlab} built-in function {\sf lsqr} and the LSMR algorithm as our test iterative solvers.
We had written the {\sc Matlab} code of LSMR by following {\sf lsqr}.

\setlength{\tabcolsep}{12pt} 
\renewcommand{\arraystretch}{1.5} 
\begin{table}[htbp]
  \centering
  \caption{Test matrices} 
  \label{tab:results}
  \begin{tabular}{|>{\centering}m{3cm}|c|c|c|c|} 
    \hline
    $A$ & $m$ & $n$ & $nnz$ & $\kappa(A)$ \\
    \hline
    illc1033 & 1033   & 320 & 4719 & 1.8888e+04 \\
    \hline
    Kemelmacher &  28452  & 9693 &   100875 &  2.3824e+04 \\
    \hline
    mesh\_deform & 234023    &  9393 &  853829 & 1.1666e+03 \\
    \hline
    photogrammetry &  1388   &  390 &    11816 &  4.3519e+08 \\
    \hline
    well1850 &  1850   &    712 & 8755 &   111.3129 \\
    \hline
    photogrammetry2 & 4472  &936 &  37056 & 1.3391e+08 \\
    \hline
  \end{tabular}
\end{table}

We use several full column rank matrices from the SuiteSparse matrix
collection \cite{Suitsparse} as $A$, and construct $b=Ax_{ls}-r_{ls}$, where the residual $r_{ls}=10^{-3}t/\|t\| $, and the solution $x_{ls}$ and
the vector $t$ are generated in the normal distribution $\mathcal{N}(0,1)$.
\Cref{tab:results} lists the test matrices $A$ and
some of their properties, where $nnz$ is the total number of nonzero entries of $A$ and
$\kappa(A)$ is the available condition number of $A$
from the SuiteSparse matrix collection. All the rows $d$ of the embedding matrices
$S$ in \cref{fig:fig1}--\ref{fig:fig12} are set to twice the number of columns of $A$, i.e., $d=2n$.

\begin{figure}[htbp]
    \centering
    \includegraphics[width=7.5cm, height=6cm]{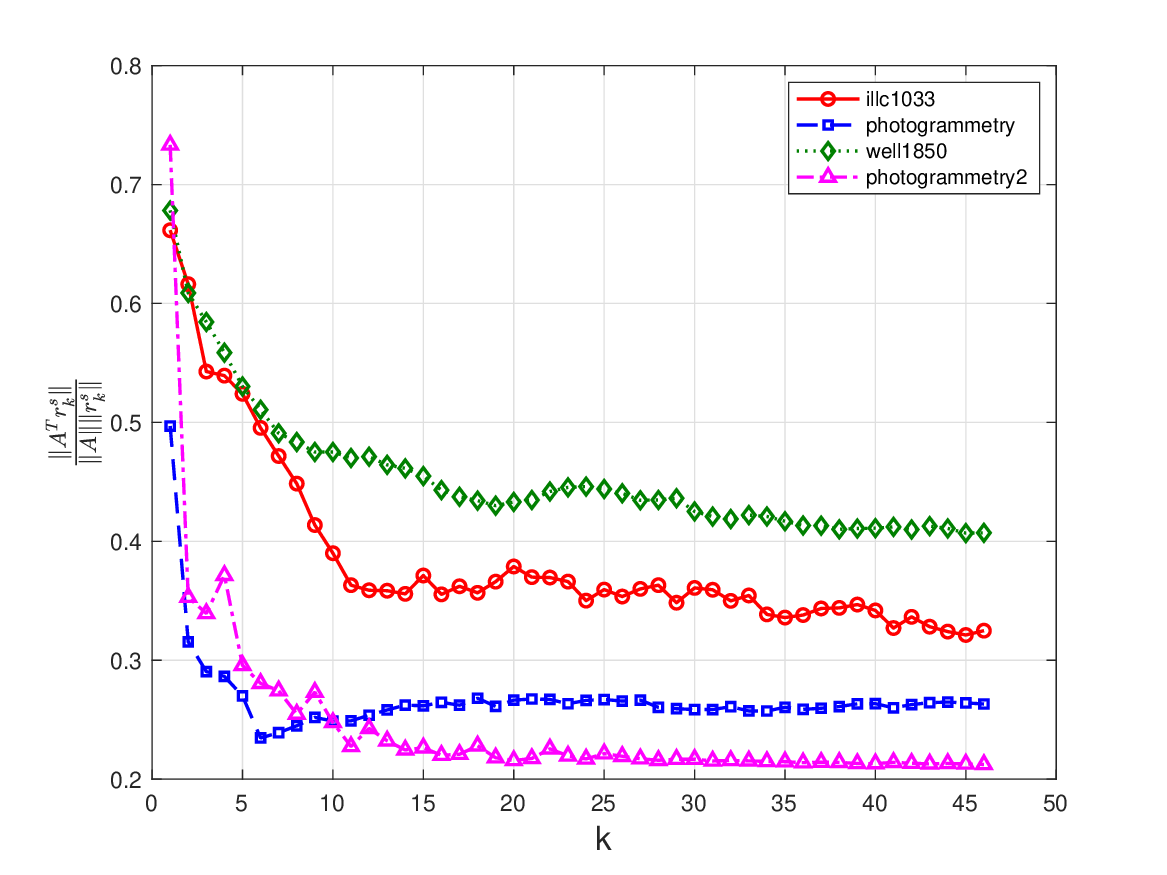} 
    \hfill
    \includegraphics[width=7.5cm, height=6cm]{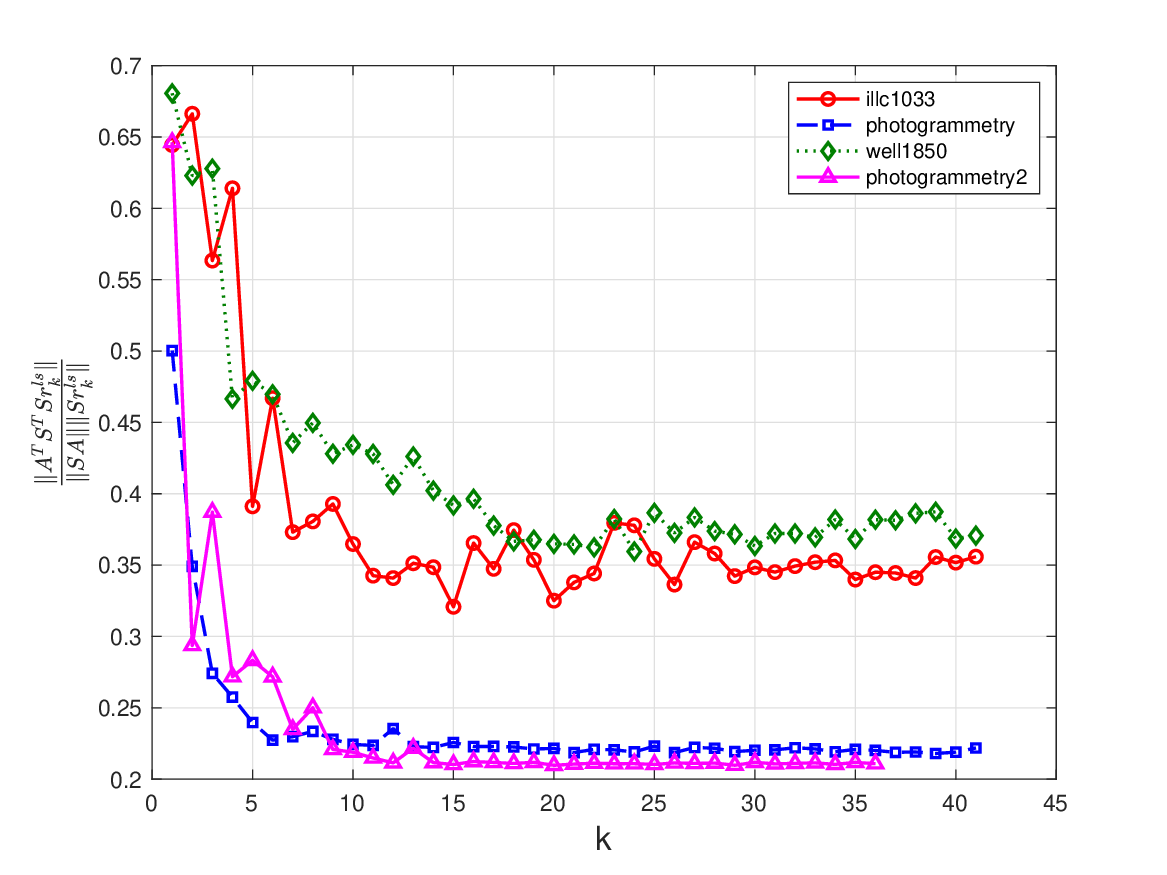}
    \caption{Relative residual norms of the normal equations of the LS problems
    \eqref{eq:original problem} and \eqref{eq:sketched problem}
 using Gaussian embedding matrices and LSQR}\label{fig:fig1}
\end{figure}

We first use illc1033, photogrammetry, well1850 and photogrammetry2 to conduct numerical experiments using
Gaussian embedded matrices. The other test matrices are excluded because generating its Gaussian embedding matrix exceeds the allocated memory of our computer. We use LSQR to compute
the $k$-step iterates $x_k^s$ for \eqref{eq:sketched problem} and
$x_k^{ls}$ for \eqref{eq:original problem}, and define the residual
$r_k^{ls}=Ax_k^{ls}-b$. \Cref{fig:fig1} shows that the convergence
curves of the left-hand sides of \eqref{eq:normal_residual_upper_bound2} and the
the counterparts by taking the $x_k^{ls}$ as approximate solutions of \eqref{eq:sketched problem}. As is seen,
LSQR undergos a rapid initial decrease followed by gradual stabilizations as iterations proceed, though
possibly irregular in the stabilization phase.
This behavior persists regardless of whether problem \eqref{eq:original problem} is approximated by problem
\eqref{eq:sketched problem} or problem \eqref{eq:sketched problem} is approximated by problem
\eqref{eq:original problem}. Strikingly, each of the rapid decreasing processes
took only dozens of iterations, indicating that LSQR
solved \eqref{eq:original problem} very quickly by solving its sketched surrogate \eqref{eq:sketched
problem} to obtain an optimal approximate solution of \eqref{eq:original problem}.
These results are in accordance with Theorem~\ref{thm:embedded_stopping} and the
key arguments between \eqref{eq:normal_residual_upper_bound2} and
\eqref{eq:normal_residual_upper_bound3}, illustrating that
the stopping criterion \eqref{stopcrit1} worked reliably when the right-hand sides lie in
$[0.99,1.01]$. Furthermore, based on bound \eqref{eq:normal_residual_upper_bound},
from \Cref{fig:fig1} (left), we can deduce that the four $\epsilon$'s are smaller than
0.42, 0.34, 0.24, and 0.225 for well1850, illc1033, photogrammetry, and photogrammetry3, respectively.

Next we examine SRHT embedding matrices.
Traditional construction of SRHT matrices via Hadamard matrices constrains the row number $m$
of $A$ to the powers of 2.
In practice, we can pad the input data with zeros to adjust the row number $m$ of $A$,
thereby rounding $m$ up to the smallest power of two that is greater than $m$ \cite{balabanov2023block}
Due to the lack of the allocated memory, mesh{\_}deform is excluded in the experiments.
\Cref{fig:fig2} exhibits the behavior of LSQR similar
to that in Figure \ref{fig:fig1}.
Therefore, the comments and observations on Figure \ref{fig:fig1} apply to \ref{fig:fig2}.

\begin{figure}[htbp!]
    \centering
    \includegraphics[width=7.5cm, height=6cm]{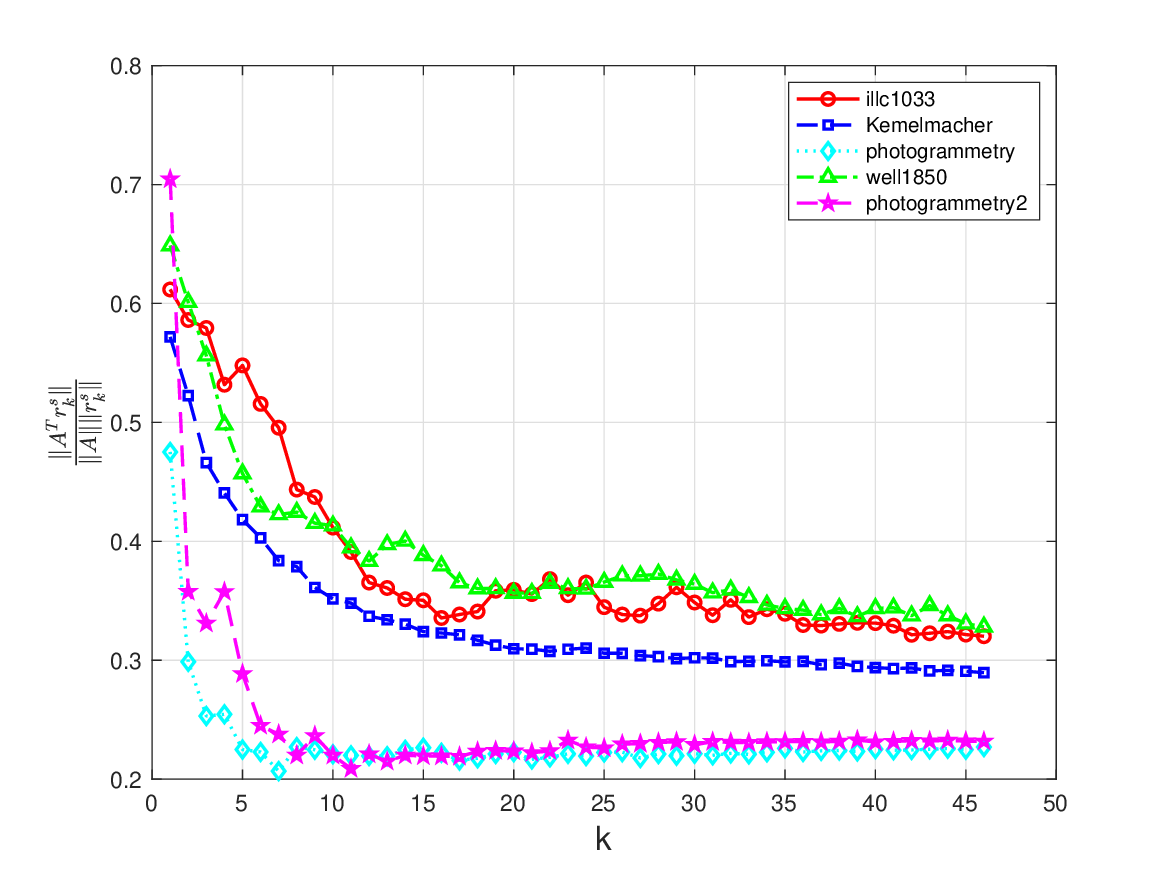} 
    \hfill
    \includegraphics[width=7.5cm, height=6cm]{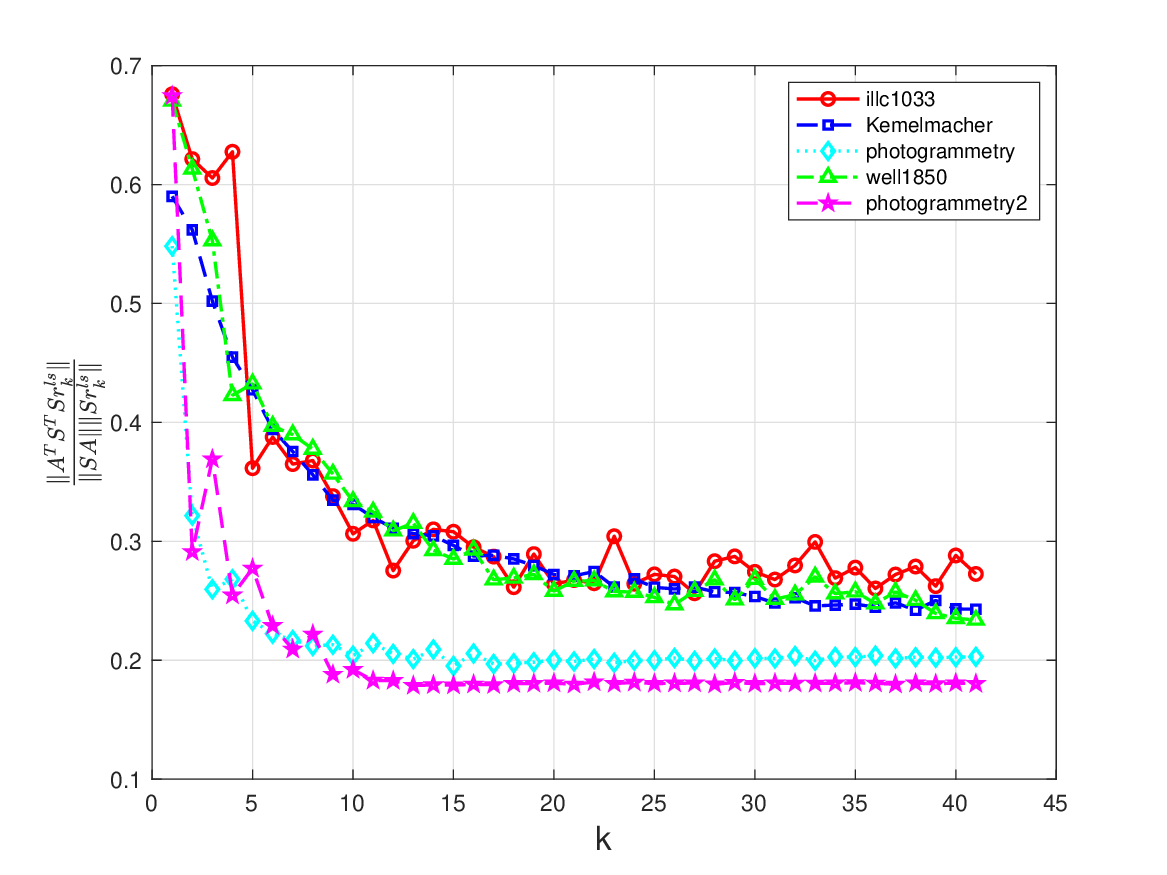}
    \caption{Relative residuals norms of the normal equations of the LS problems
    \eqref{eq:original problem} and \eqref{eq:sketched problem} using SRHT embedding matrices and LSQR}\label{fig:fig2}
\end{figure}


\begin{figure}[htbp!]
    \centering
    \includegraphics[width=7.5cm, height=6cm]{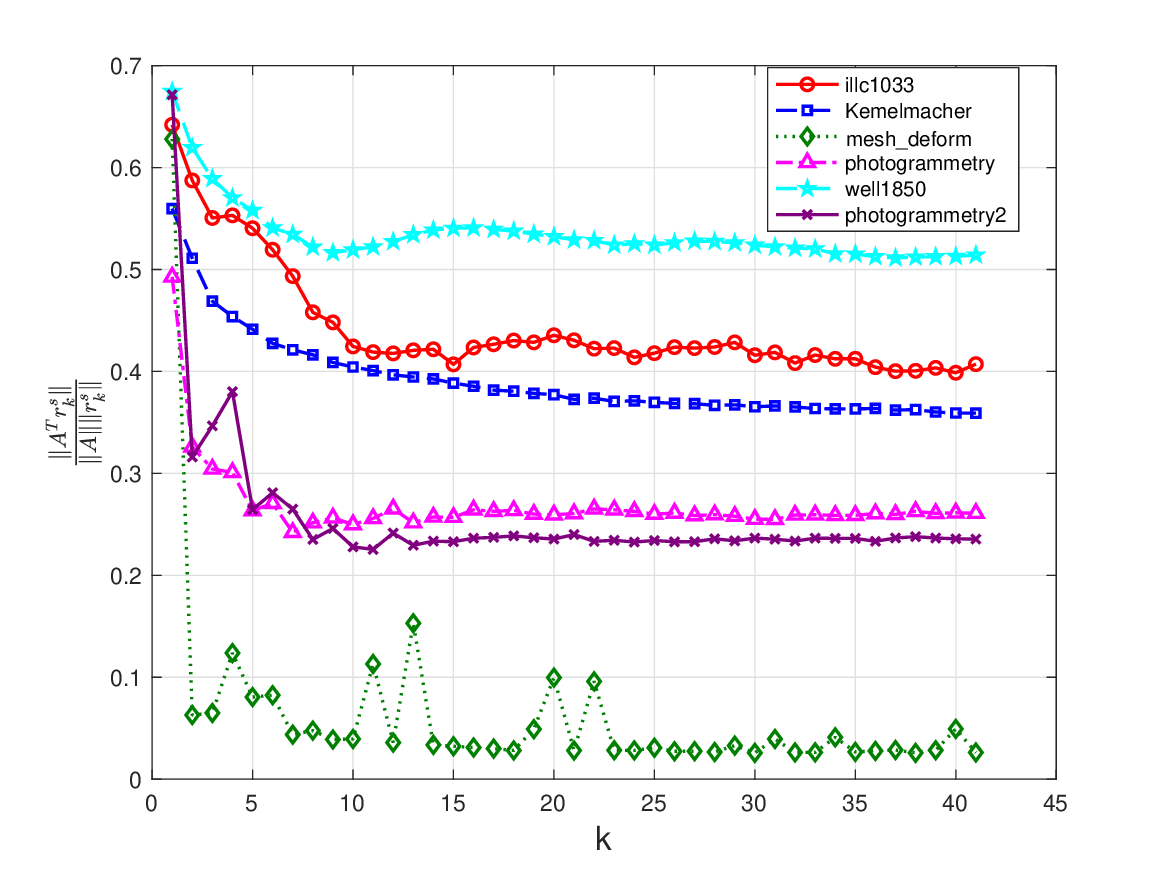} 
    \hfill
    \includegraphics[width=7.5cm, height=6cm]{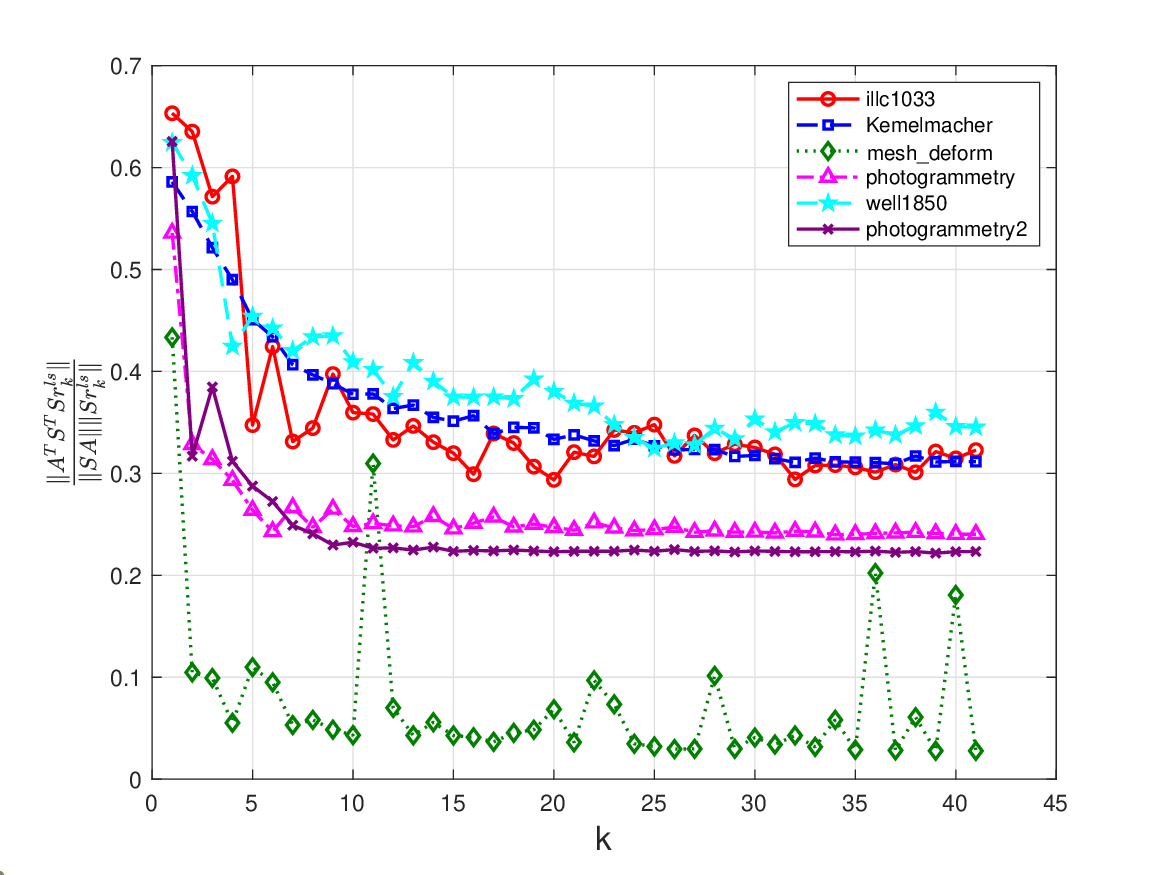}
    \caption{Relative residual norms of  the normal equations of the LS problems
    \eqref{eq:original problem} and \eqref{eq:sketched problem}
 using sparse embedding matrices and LSQR}\label{fig:fig3}
\end{figure}

\Cref{fig:fig3} depicts the convergence processes of LSQR for \eqref{eq:original problem} and \eqref{eq:sketched problem} using sparse embedding matrices. We observe similar and some dissimilar
phenomena to those in \Cref{fig:fig1,fig:fig2}.
Overall, the sparse embedding did not work as well as Gaussian and SRHT embeddings, and
exhibited more irregular behavior and was less effective, particularly for illc1033 and well1850.
The less effectiveness is expected from our previous description and discussion
in \cref{subsec:compare} on these three kinds
of embedding techniques when taking the same embedding row number $d$. Irregular convergence behavior is
not unusual because of the convergence
property of LSQR residual norms of the normal equations, as we have explained in \cref{stopdesign}. Therefore,
the stopping criterion \eqref{stopcrit1} for LSQR may indeed be unreliable.

\begin{figure}[htbp]
    \centering
    \includegraphics[width=7.5cm, height=6cm]{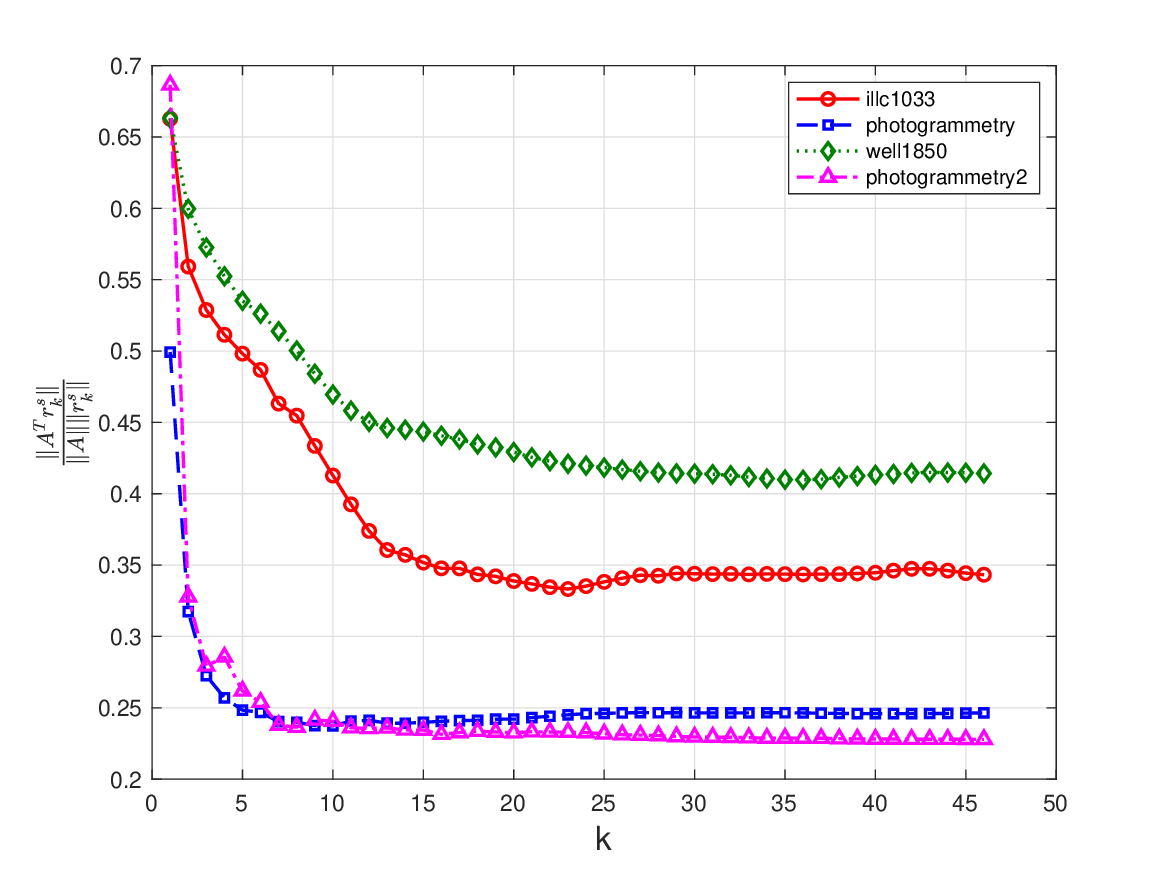} 
    \hfill
    \includegraphics[width=7.5cm, height=6cm]{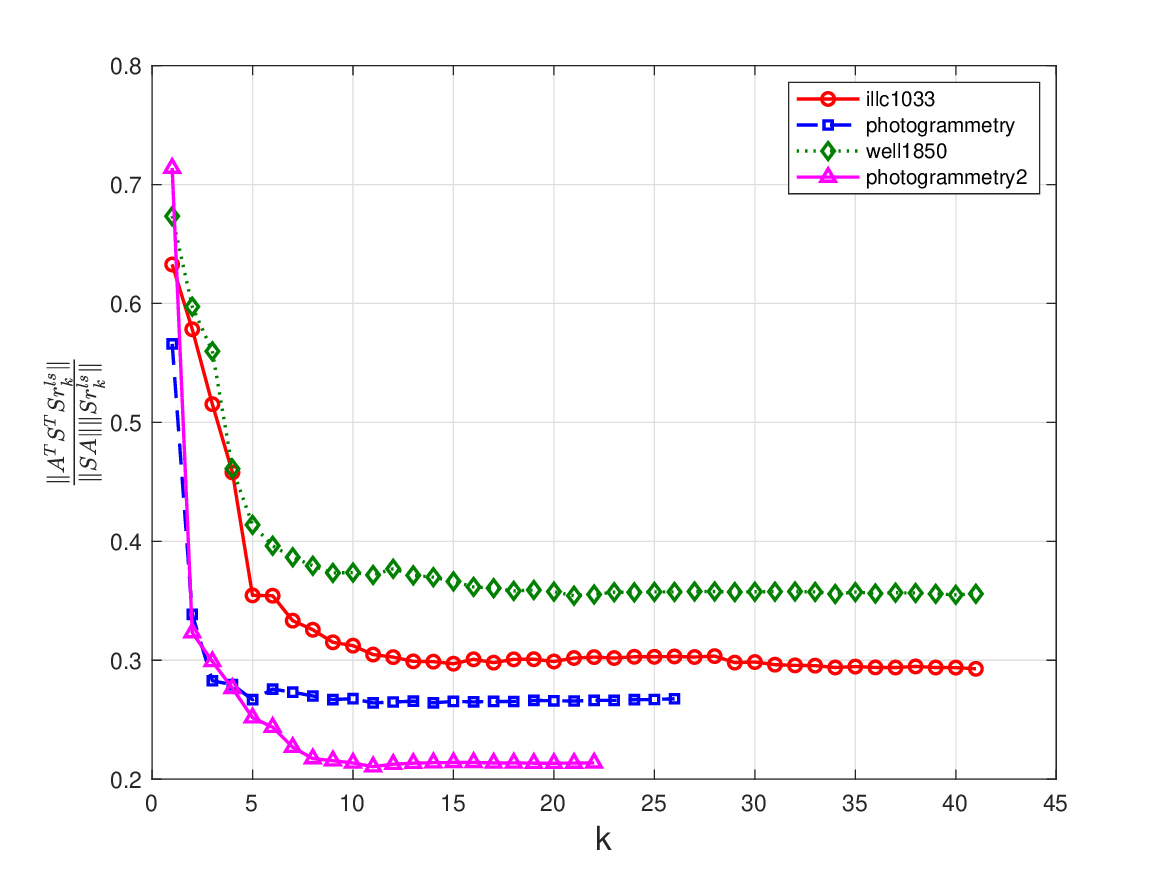}
    \caption{Relative residual norms of the normal equations of the LS problems
    \eqref{eq:original problem} and \eqref{eq:sketched problem}
    using Gaussian embedding matrices and LSMR}\label{fig:fig4}
\end{figure}

\begin{figure}[htbp]
    \centering
    \includegraphics[width=7.5cm, height=6cm]{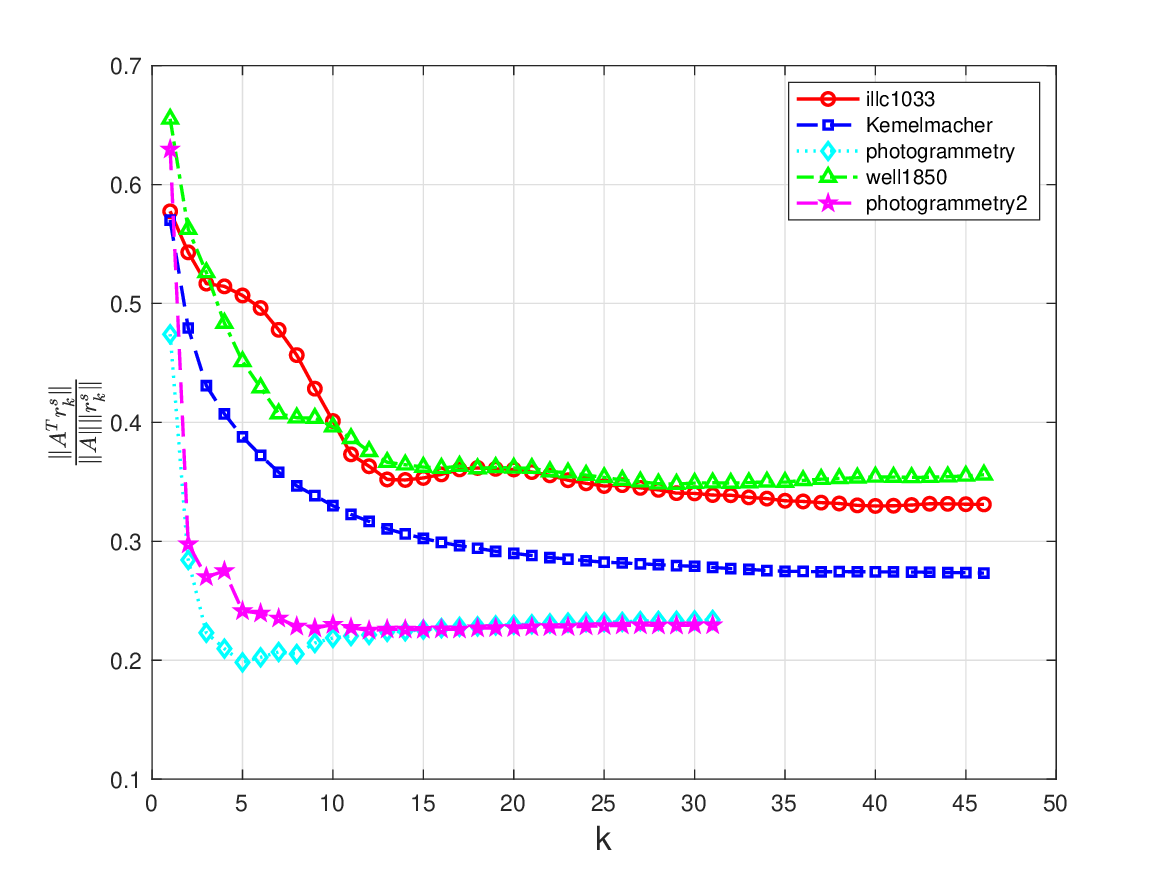} 
    \hfill
    \includegraphics[width=7.5cm, height=6cm]{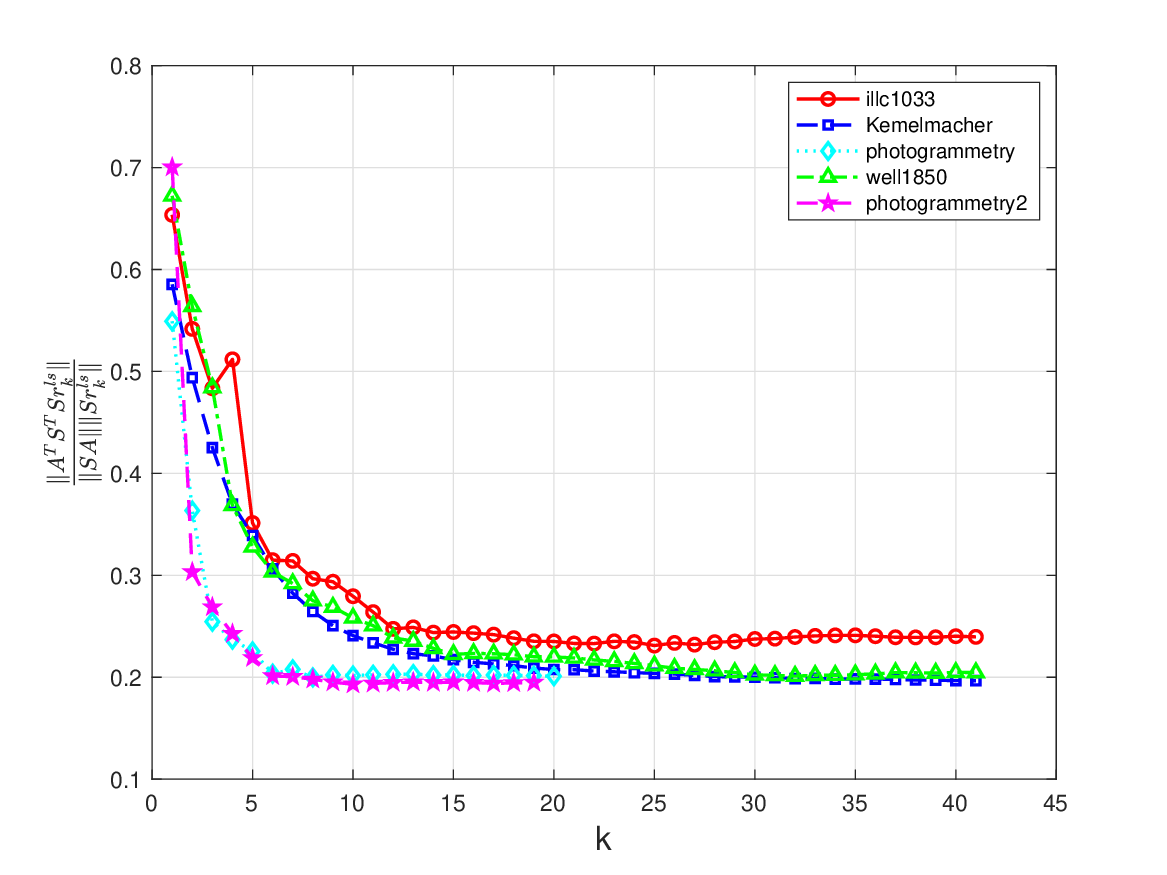}
    \caption{Relative residual norms of the normal equations of the LS problems
    \eqref{eq:original problem} and \eqref{eq:sketched problem}
    using SRHT embedding matrices and LSMR}\label{fig:fig5}
\end{figure}

\begin{figure}[htbp]
    \centering
    \includegraphics[width=7.5cm, height=6cm]{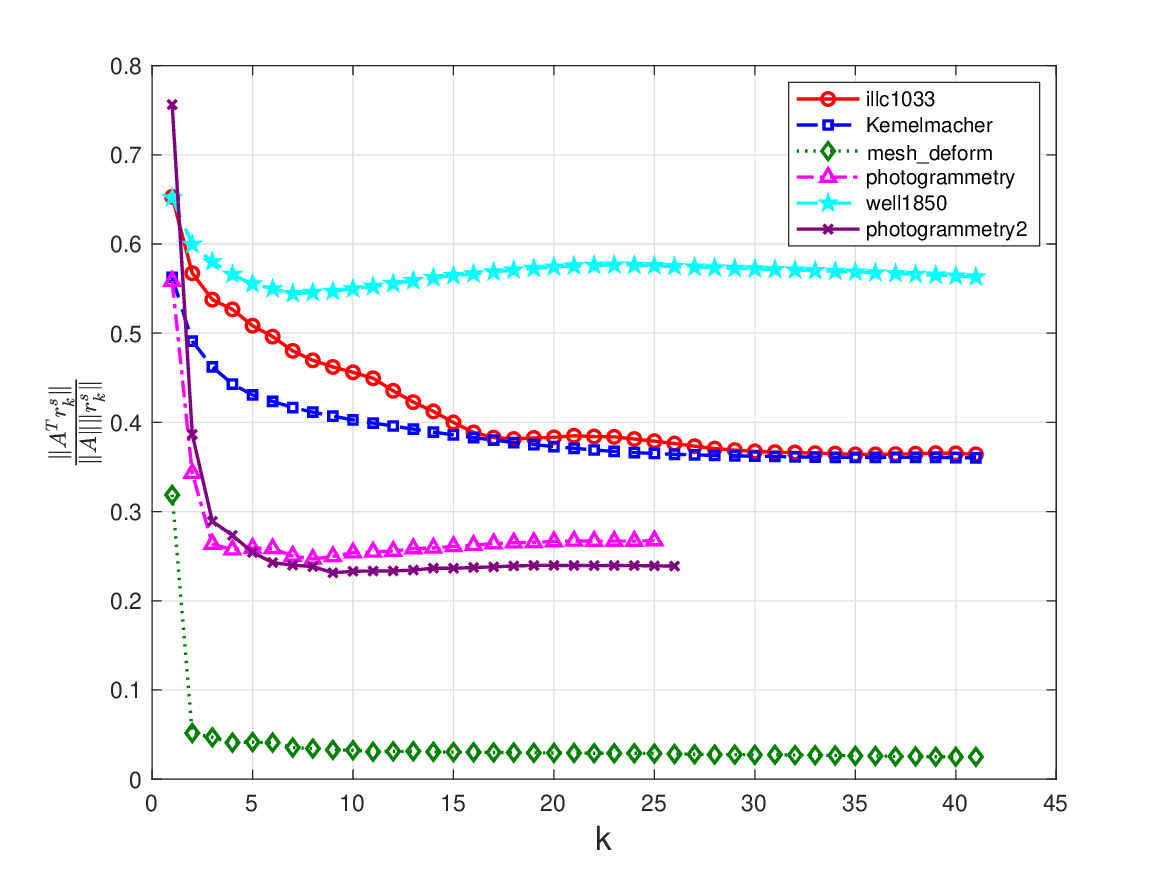} 
    \hfill
    \includegraphics[width=7.5cm, height=6cm]{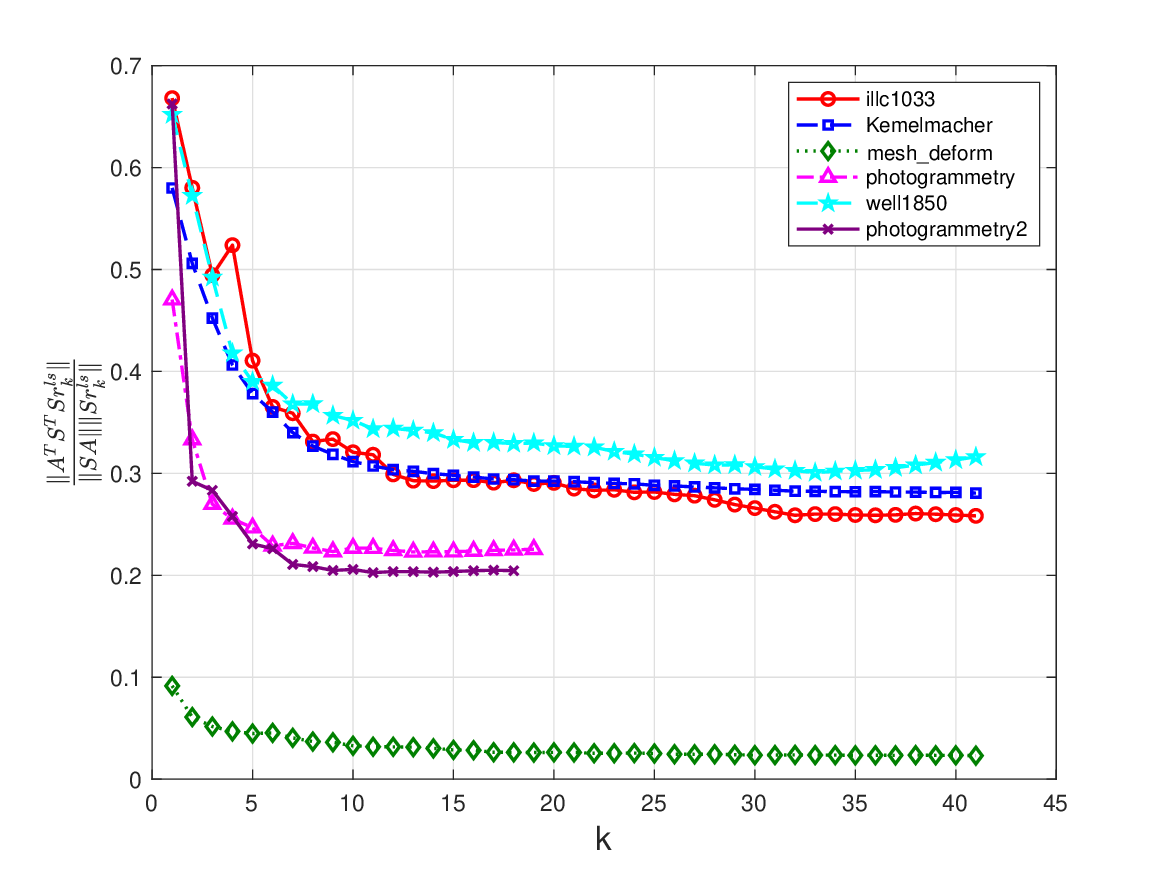}
    \caption{Relative residual norms of the normal equations of the LS problems
    \eqref{eq:original problem} and \eqref{eq:sketched problem}
   using sparse embedding matrices and LSMR}\label{fig:fig6}
\end{figure}

We conducted similar experiments  using the LSMR algorithm. Similar to \cref{fig:fig1}--\cref{fig:fig3}.
The convergence processes are drawn in Figures \ref{fig:fig4}--\ref{fig:fig6}.
Because of the monotonic decreasing property of residual norms obtained by LSMR when applied
to the normal equations of \eqref{eq:original problem} and \eqref{eq:sketched problem},
as is expected, LSMR
exhibited guaranteed monotonic convergence decrease; its convergence curves
demonstrate much smoother than those by LSQR, as is clearly seen
by comparing each of Figures \ref{fig:fig1}--\ref{fig:fig3} and the corresponding one of
Figures \ref{fig:fig4}--\ref{fig:fig6}. This justifies the general-purpose reliability and robustness of
the stopping criterion \eqref{stopcrit1} for LSMR, as we have addressed in \cref{stopdesign}.

\begin{figure}[htbp]
    \centering
    \includegraphics[width=7.5cm, height=6cm]{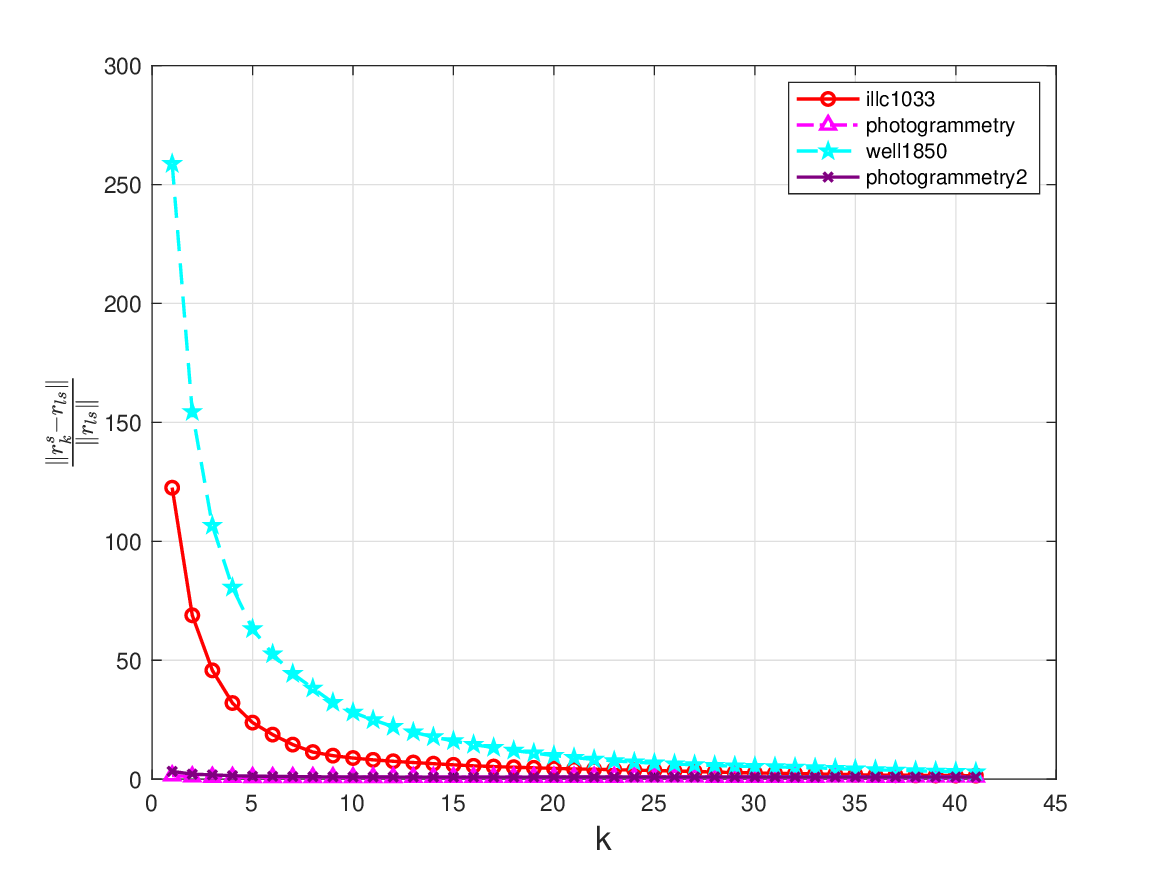} 
    \hfill
    \includegraphics[width=7.5cm, height=6cm]{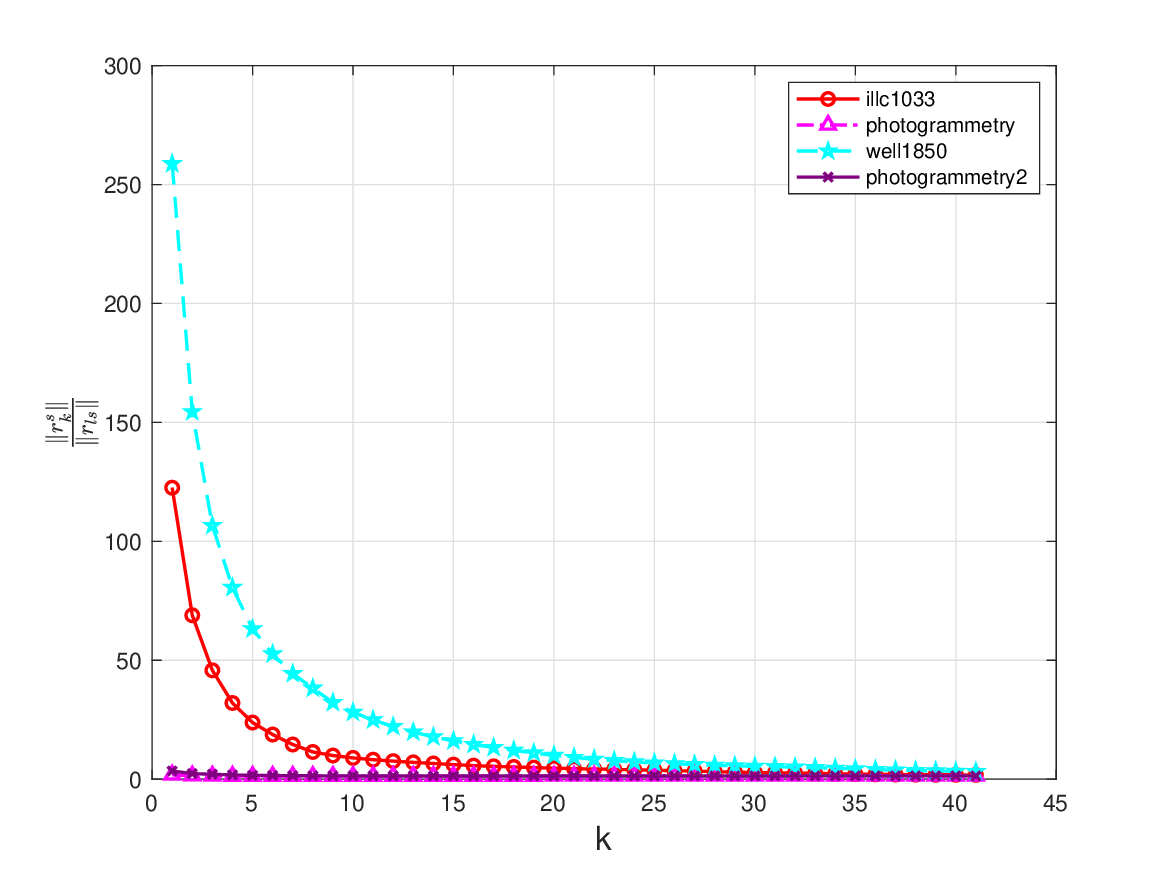}
    \caption{Residual norms $\|r_k^s\|$ using Gaussian embedding matrices and LSQR }\label{fig:fig7}
\end{figure}

\begin{figure}[htbp]
    \centering
    \includegraphics[width=7.5cm, height=6cm]{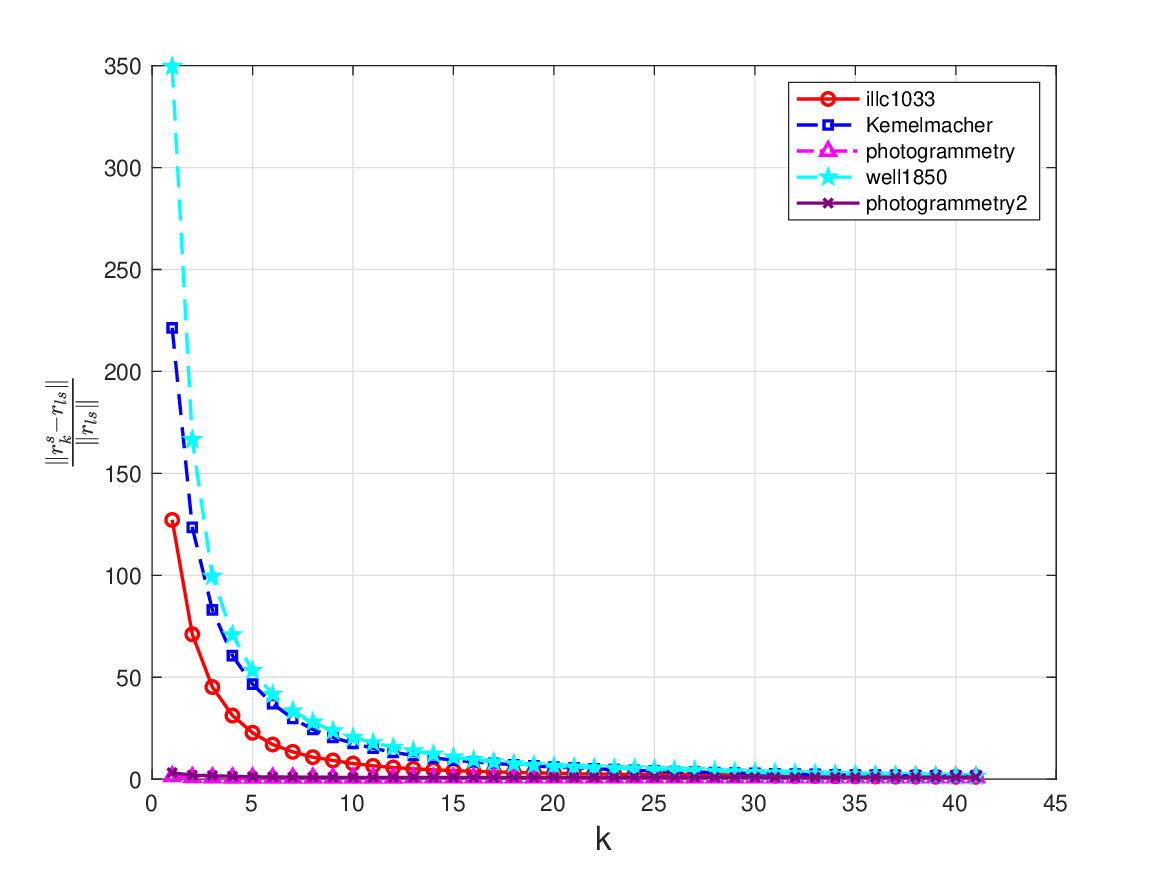} 
    \hfill
    \includegraphics[width=7.5cm, height=6cm]{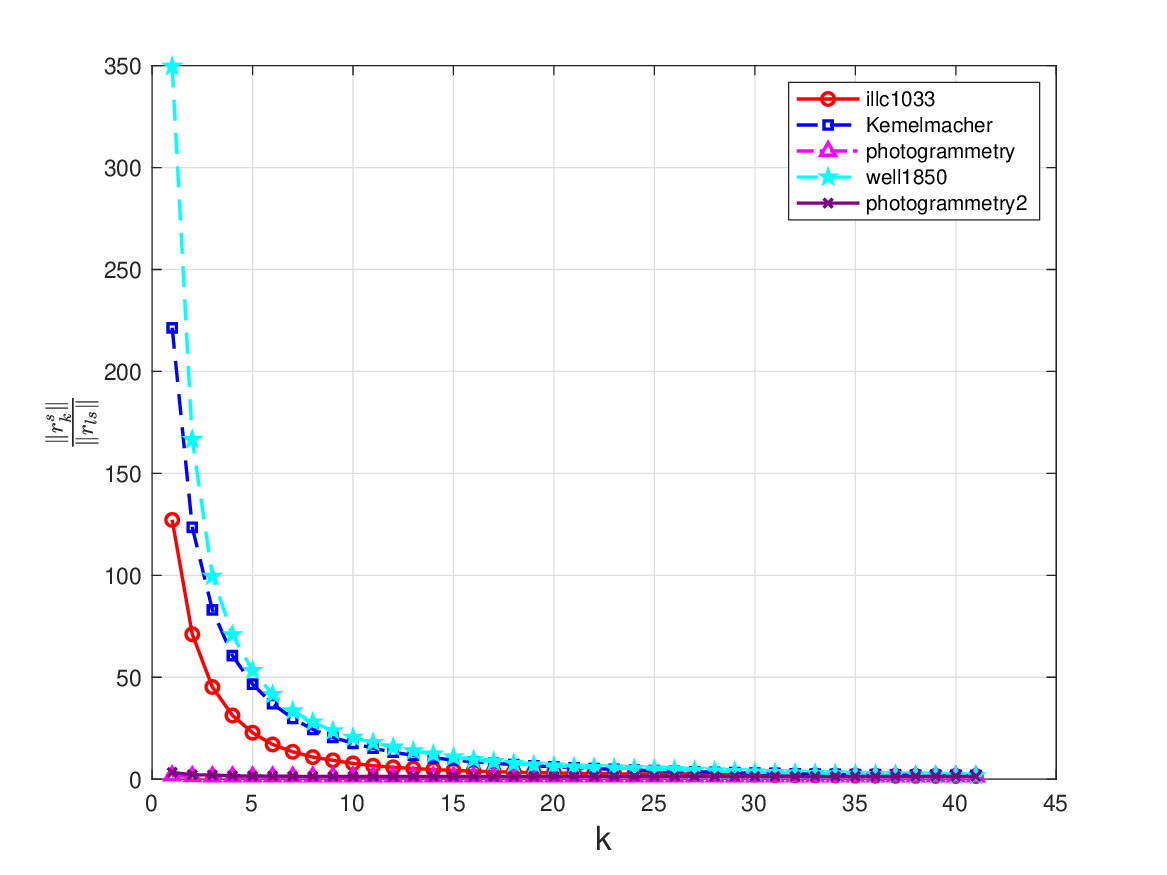}
    \caption{Residual norms $\|r_k^s\|$
   using SRHT embedding matrices and LSQR}\label{fig:fig8}
\end{figure}

\begin{figure}[htbp]
    \centering
    \includegraphics[width=7.5cm, height=6cm]{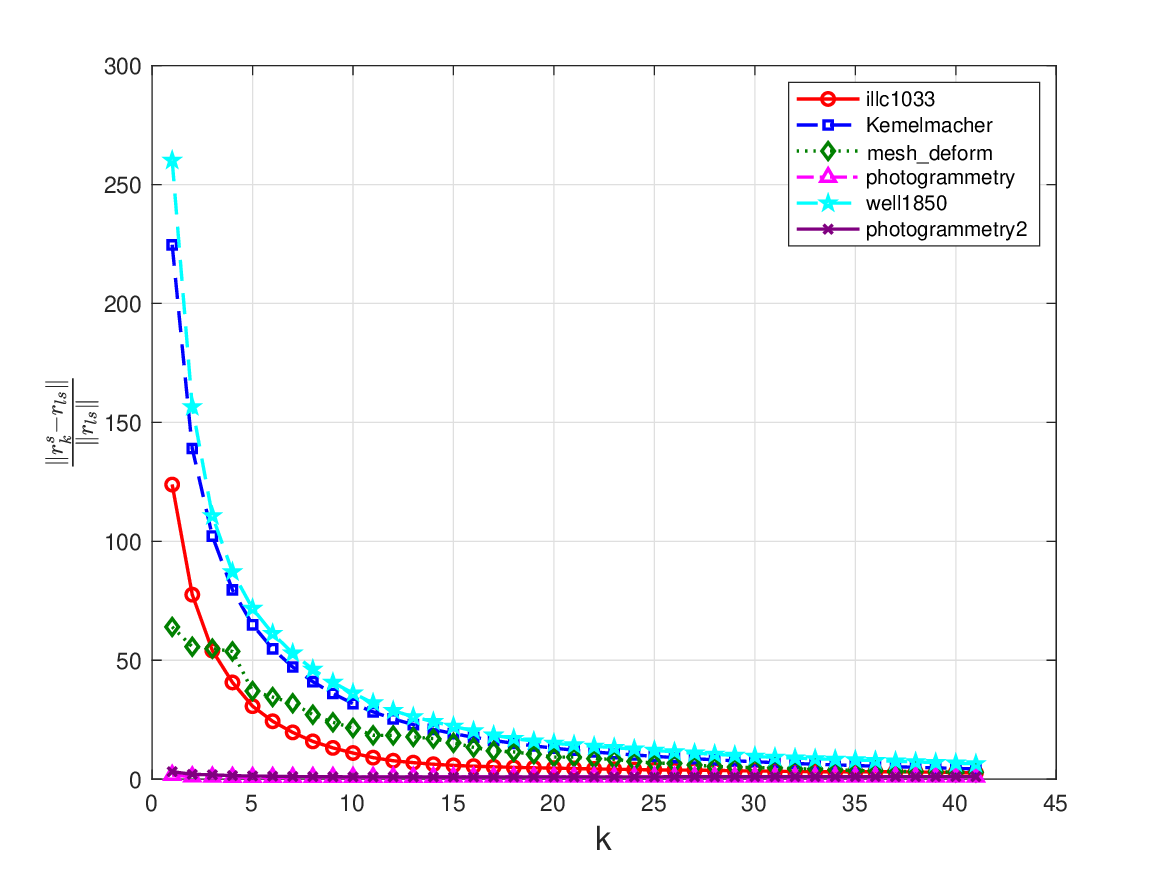} 
    \hfill
    \includegraphics[width=7.5cm, height=6cm]{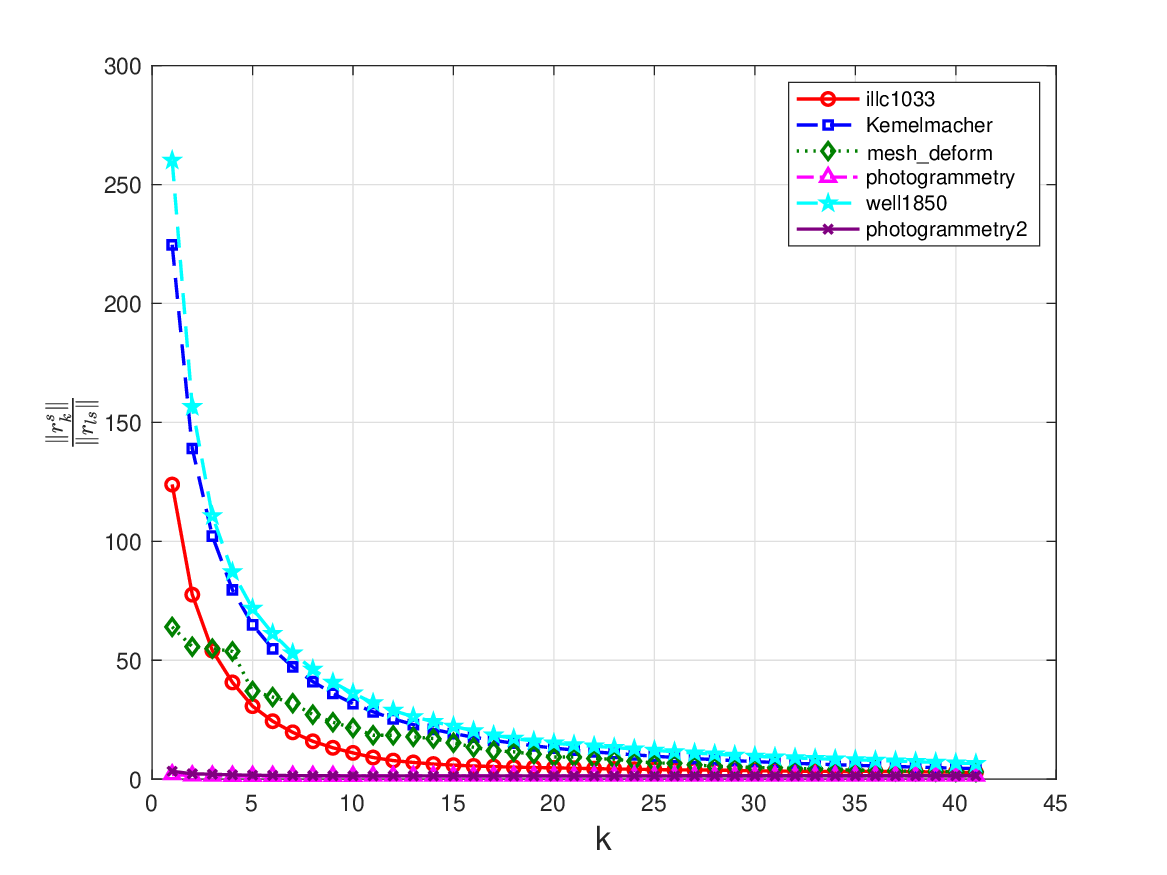}
    \caption{Residual norms $\|r_k^s\|$ using sparse embedding matrices and LSQR}\label{fig:fig9}
\end{figure}

\begin{figure}[htbp]
    \centering
    \includegraphics[width=7.5cm, height=6cm]{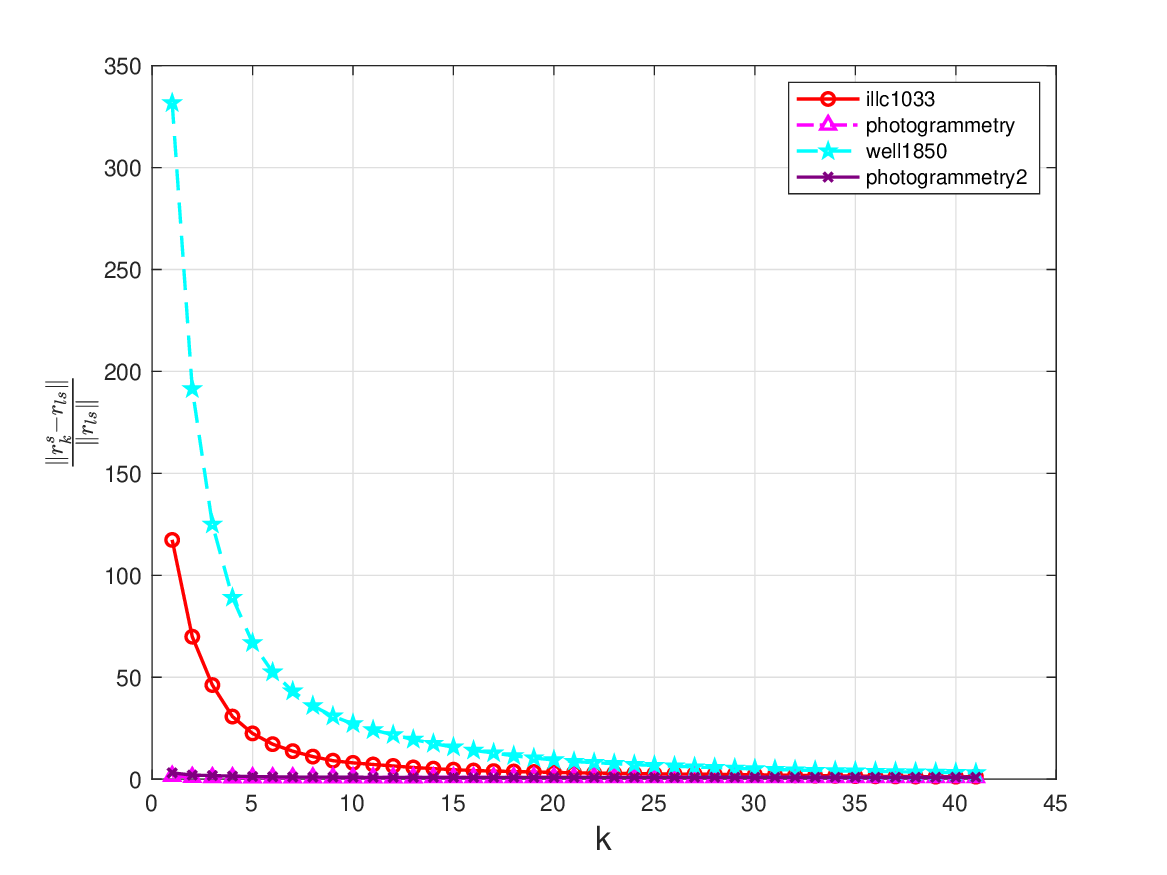} 
    \hfill
    \includegraphics[width=7.5cm, height=6cm]{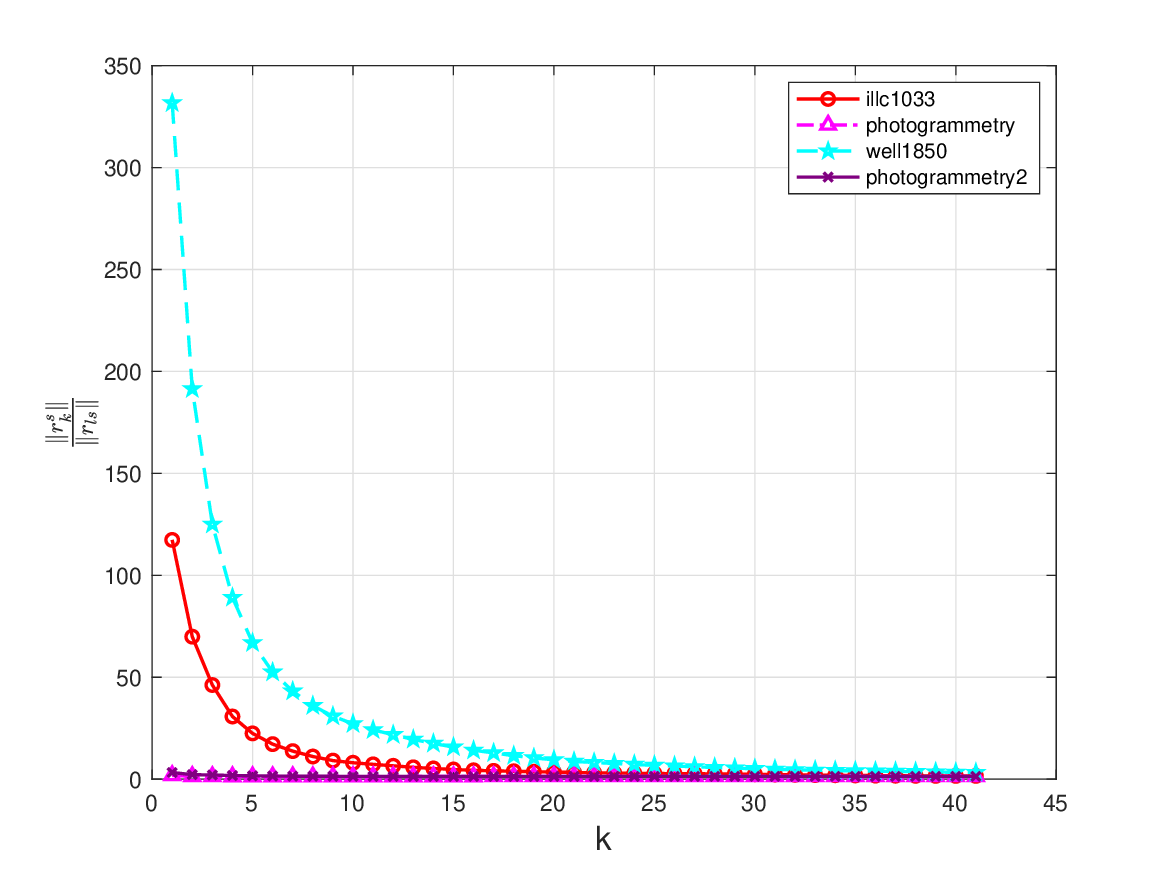}
    \caption{Residual norms $\|r_k^s\|$ using Gaussian embedding matrices and LSMR }\label{fig:fig10}
\end{figure}

\begin{figure}[htbp]
    \centering
    \includegraphics[width=7.5cm, height=6cm]{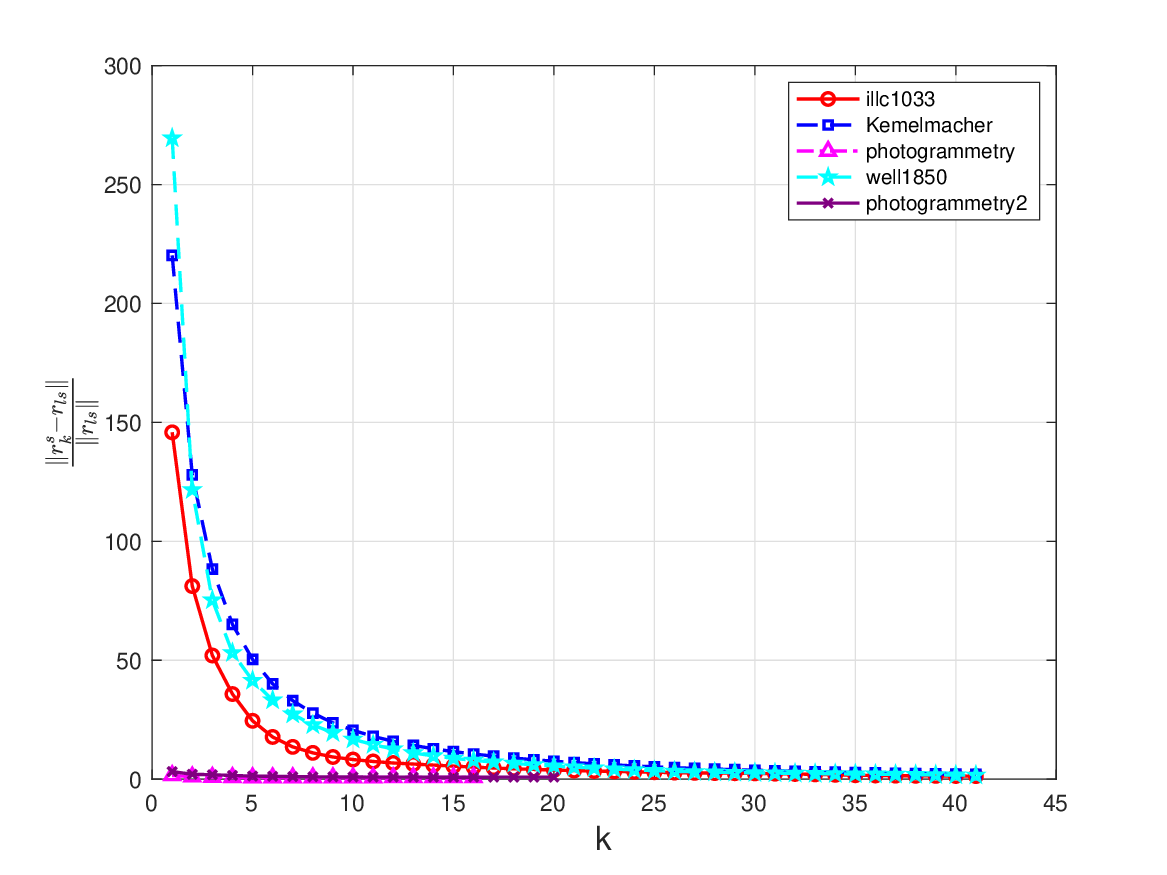} 
    \hfill
    \includegraphics[width=7.5cm, height=6cm]{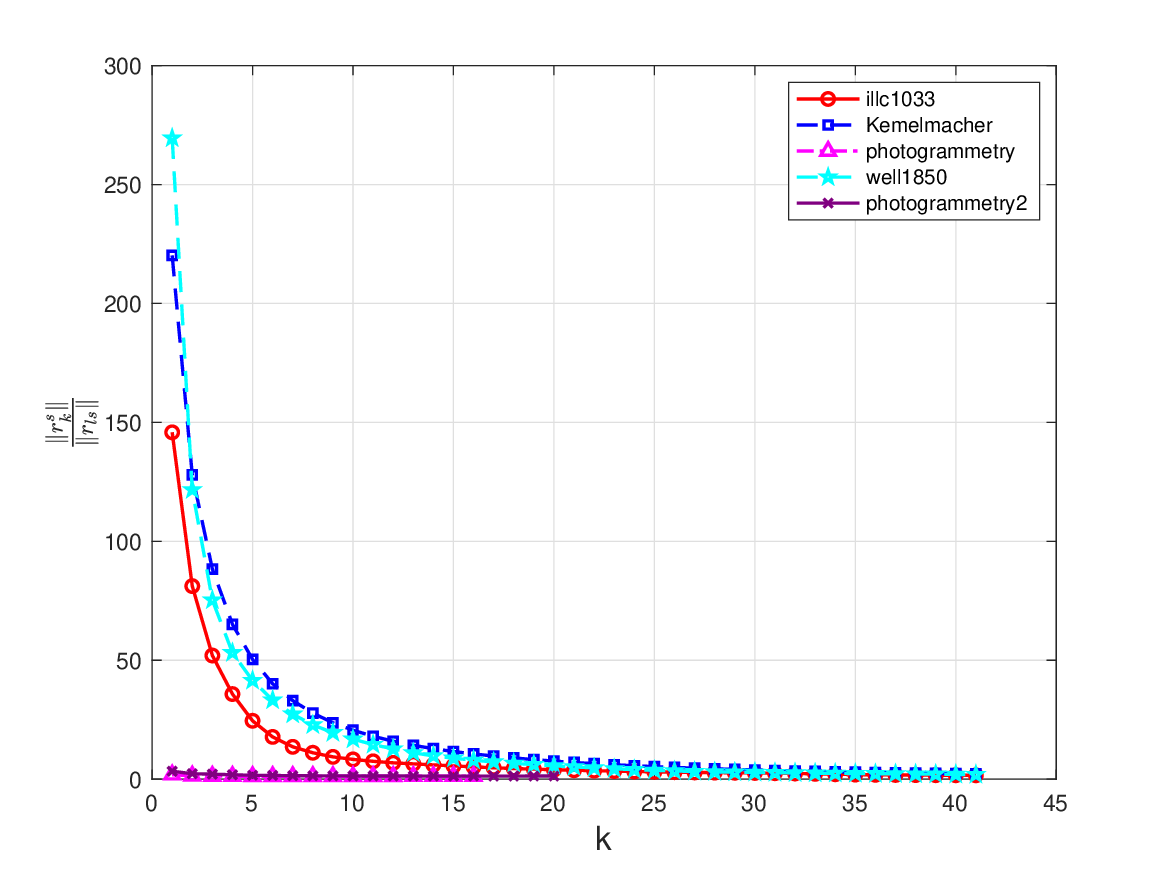}
    \caption{Residual norms $\|r_k^s\|$
   using SRHT embedding matrices and LSMR}\label{fig:fig11}
\end{figure}

\begin{figure}[htbp]
    \centering
    \includegraphics[width=7.5cm, height=6cm]{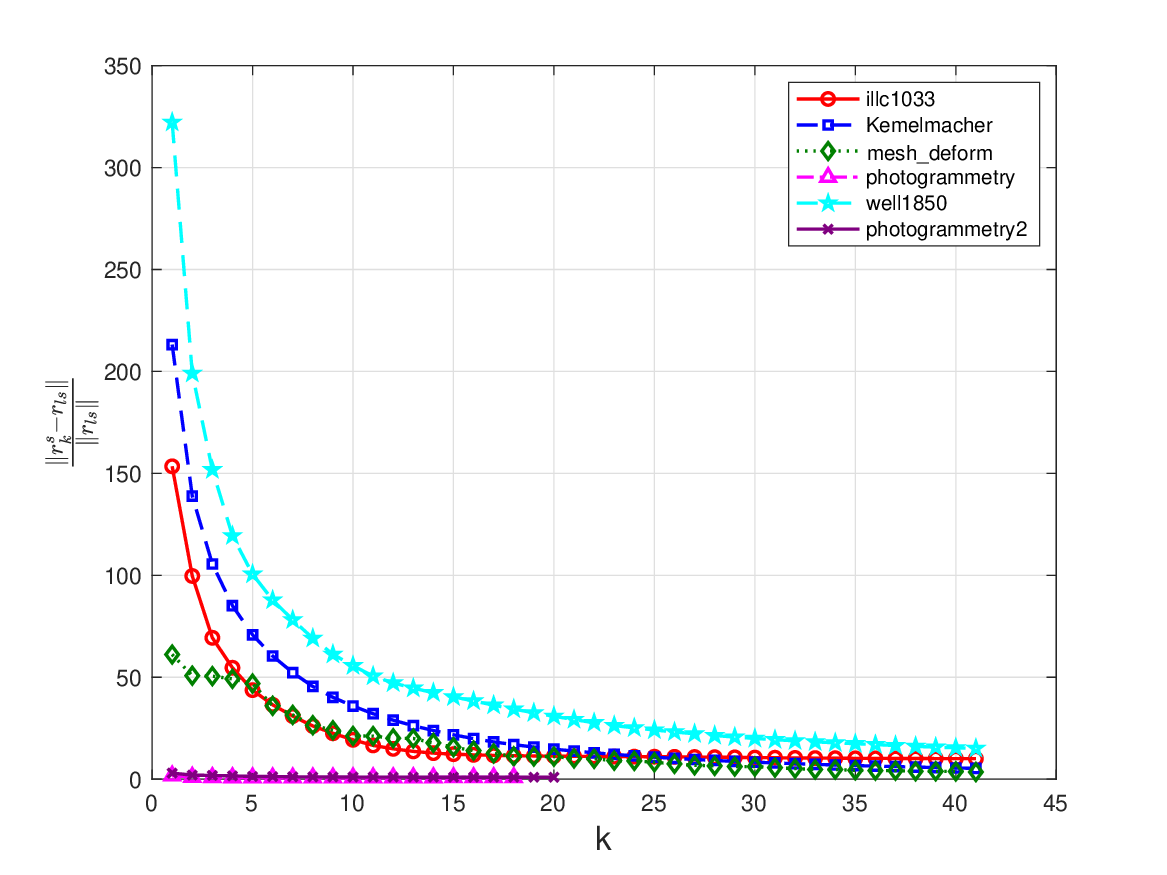} 
    \hfill
    \includegraphics[width=7.5cm, height=6cm]{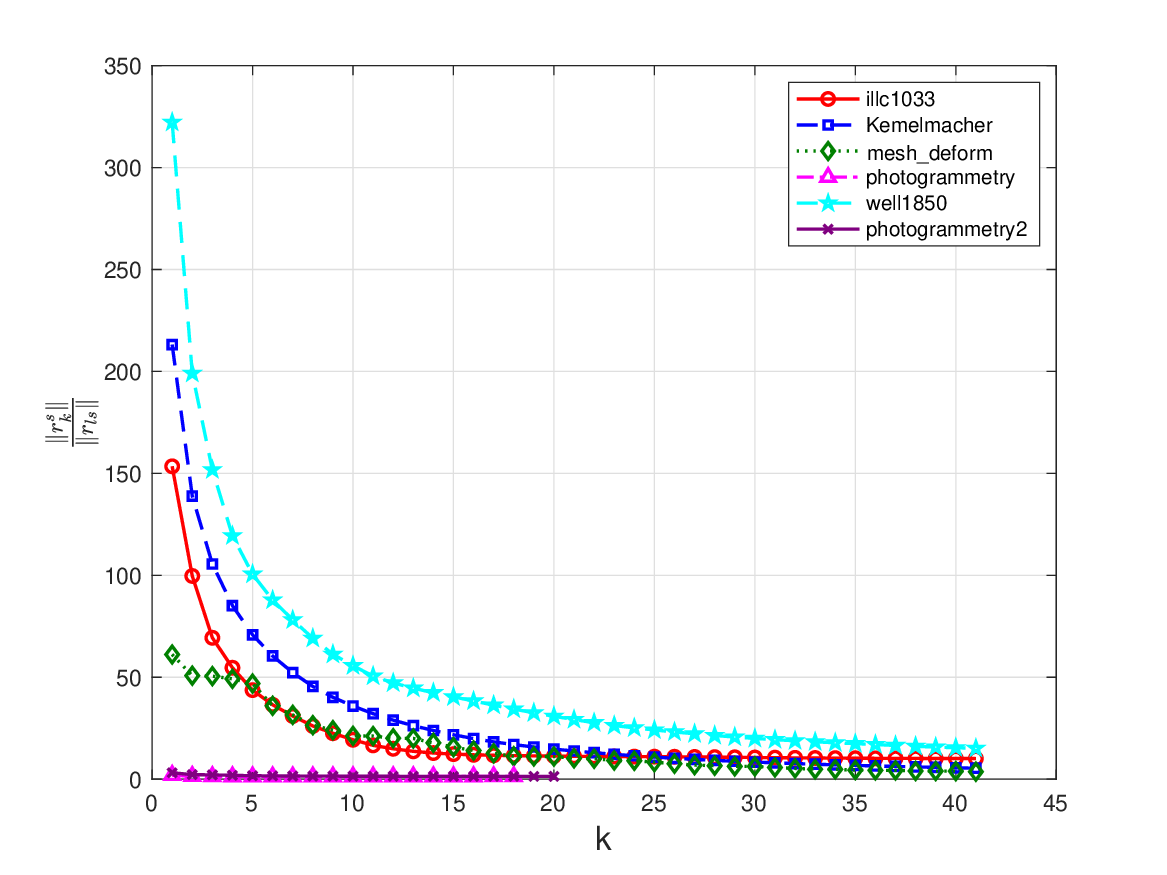}
    \caption{Residual norms $\|r_k^s\|$ using sparse embedding matrices and LSMR}\label{fig:fig12}
\end{figure}

\begin{figure}[htbp!]
    \centering
    \includegraphics[width=7.5cm, height=6cm]{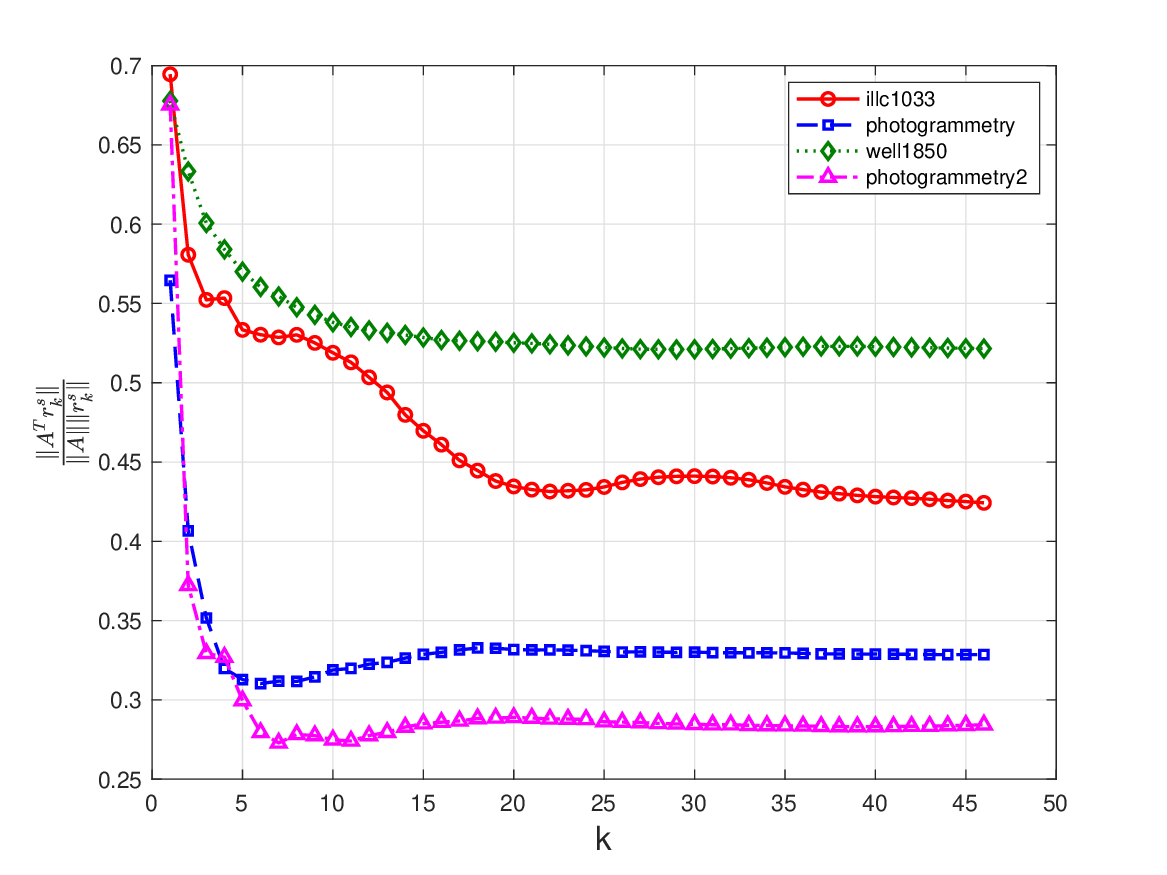} 
    \hfill
    \includegraphics[width=7.5cm, height=6cm]{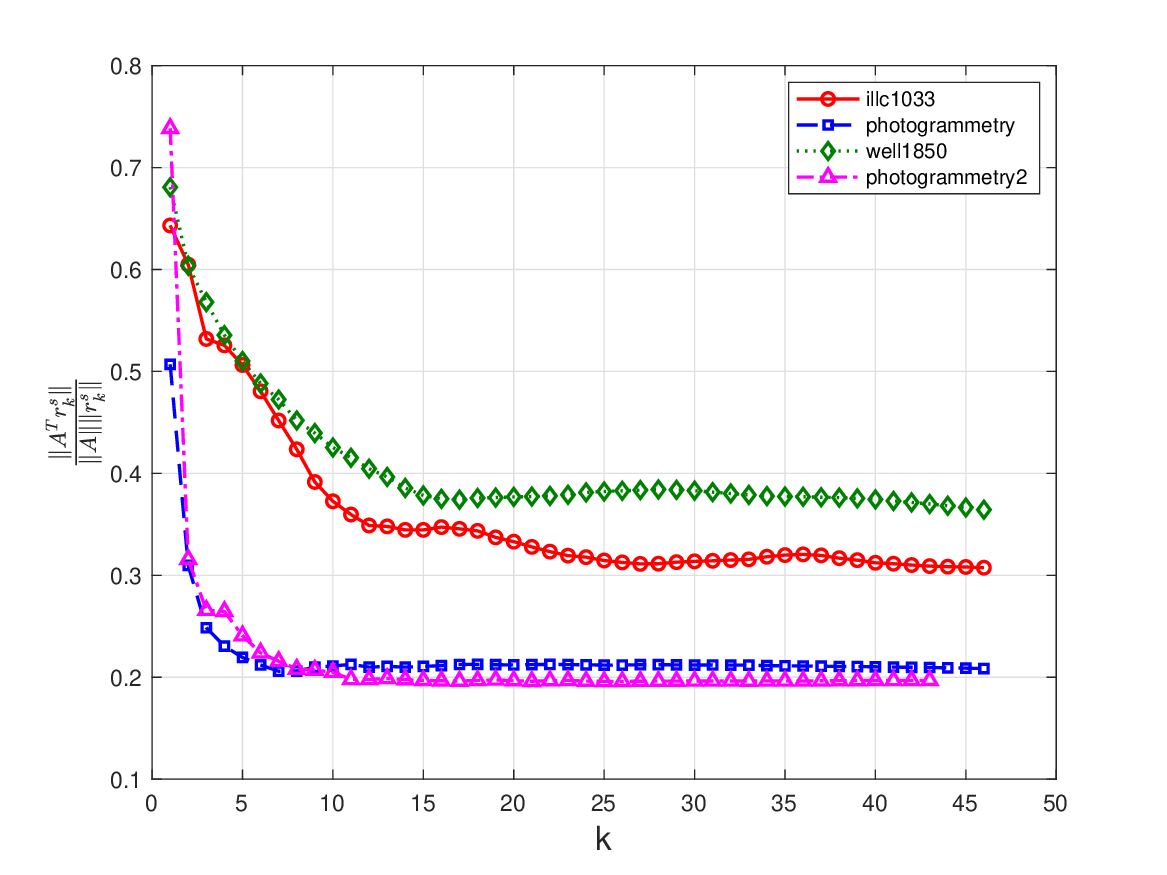}
    \caption{The convergence curves of LSMR for the problems with different rows $d$ using Gaussian embedding matrices}\label{fig:fig13}
\end{figure}

\begin{figure}[htbp!]
    \centering
    \includegraphics[width=7.5cm, height=6cm]{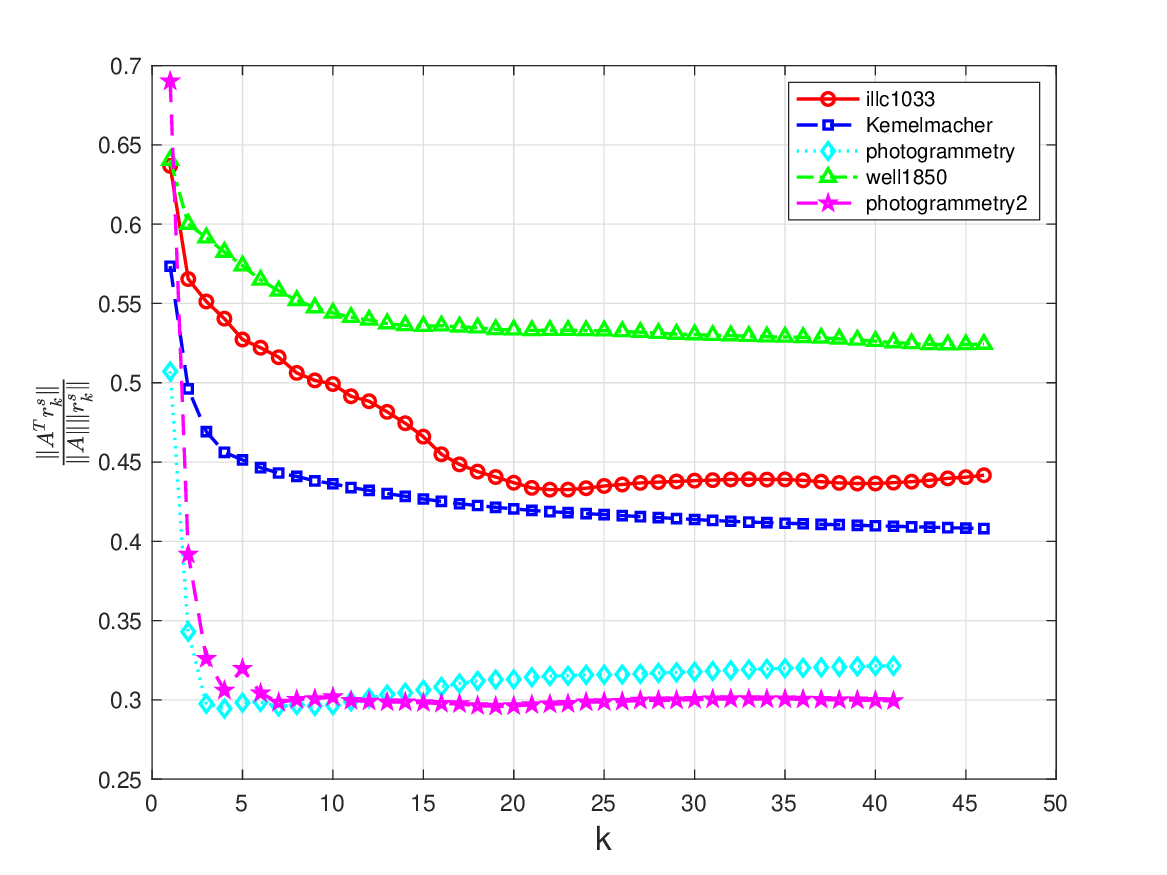} 
    \hfill
    \includegraphics[width=7.5cm, height=6cm]{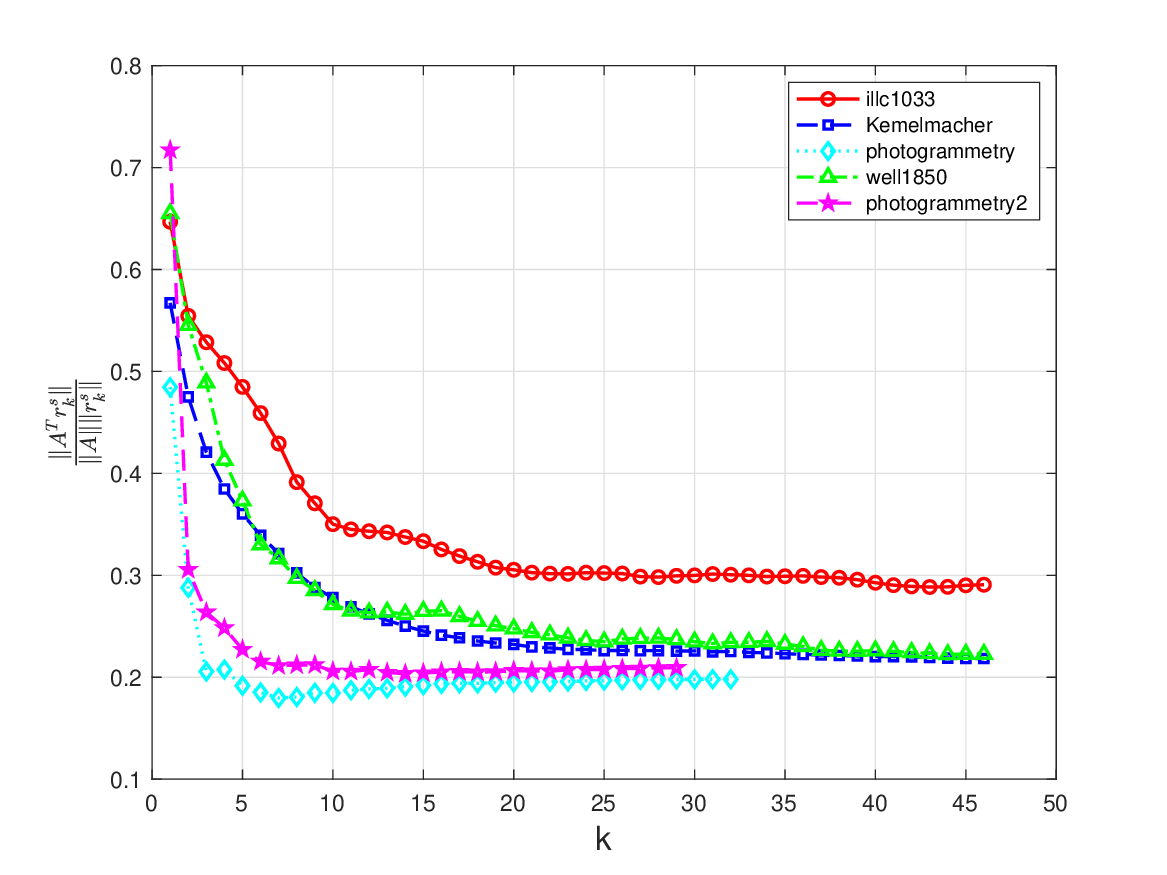}
    \caption{The convergence curves of LSMR for the problems with different rows $d$ using SRHT embedding matrices}\label{fig:fig14}
\end{figure}

\begin{figure}[htbp!]
    \centering
    \includegraphics[width=7.5cm, height=6cm]{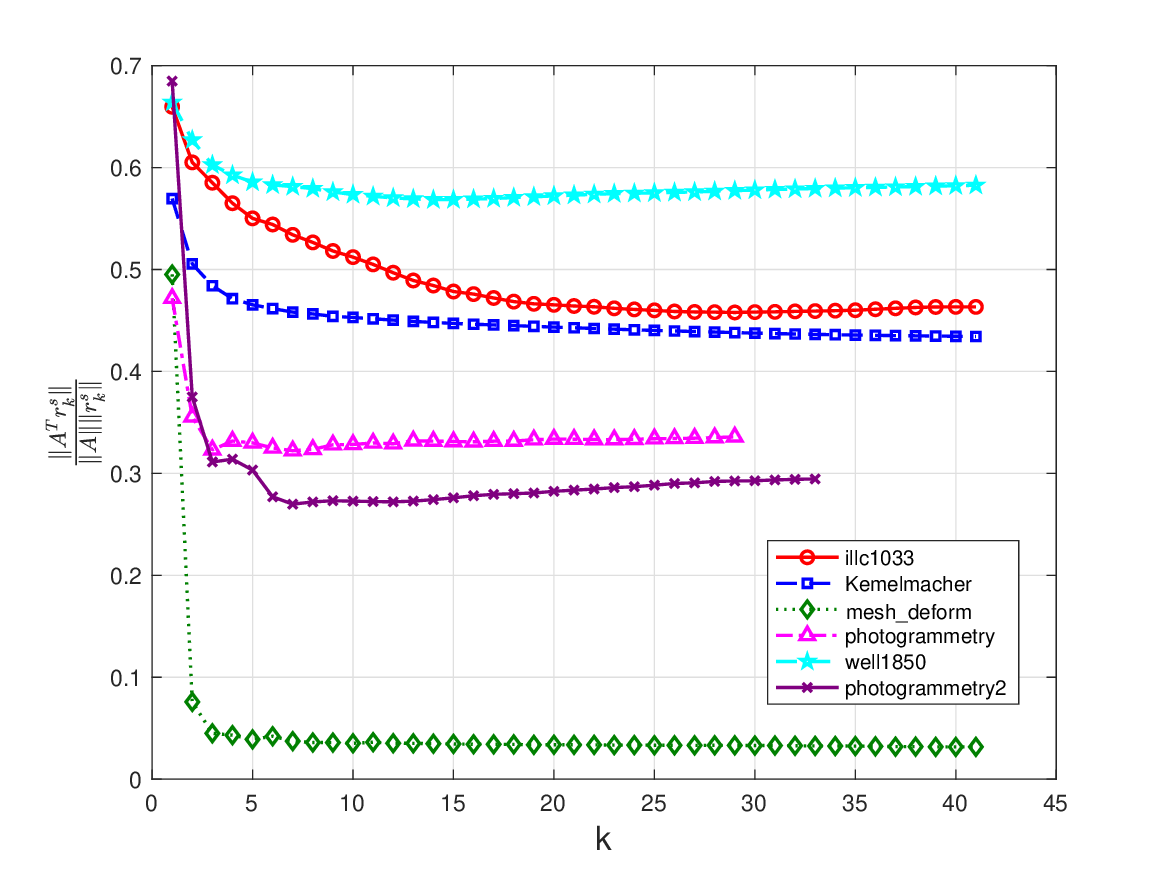} 
    \hfill
    \includegraphics[width=7.5cm, height=6cm]{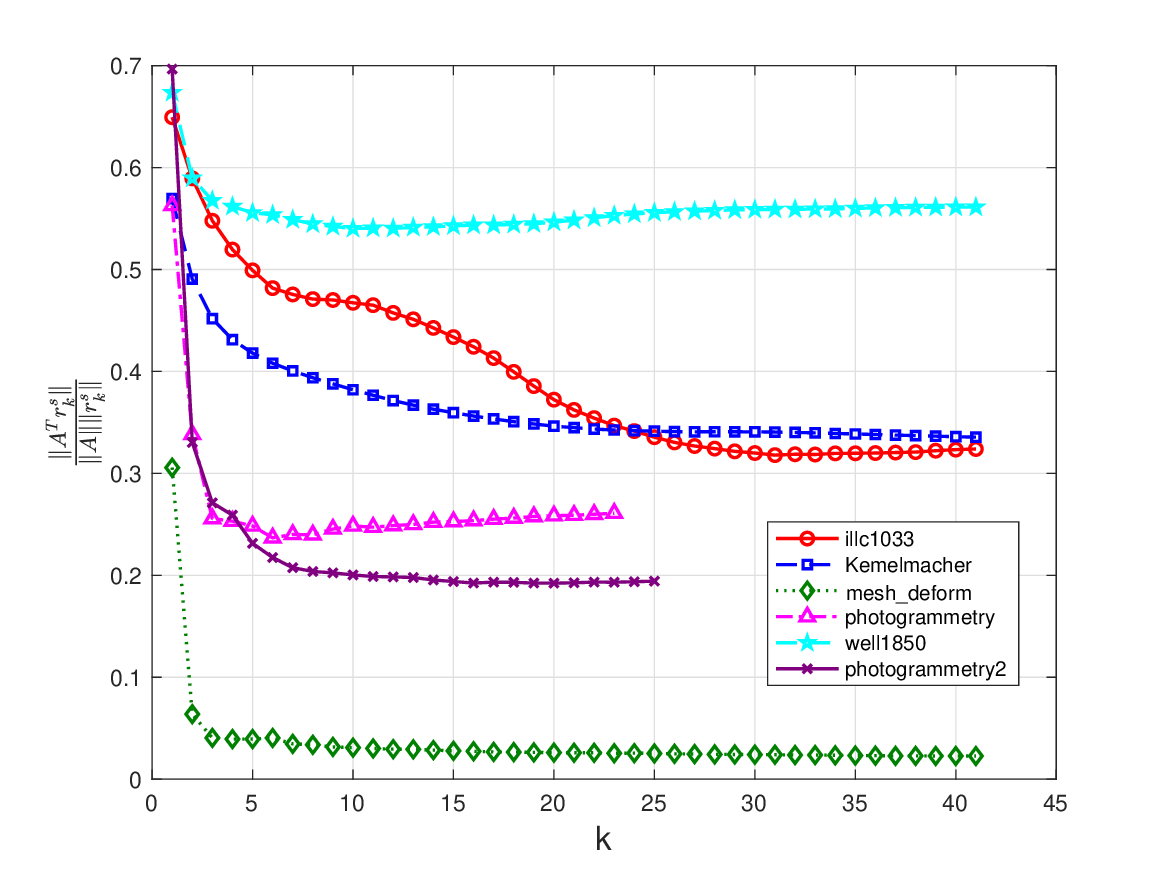}
    \caption{The convergence curves of LSMR for the problems with different rows $d$ using sparse embedding matrices}\label{fig:fig15}
\end{figure}

Now we investigate at the important quantities $\|r_k^s-r_{ls}\|/\|r_{ls}\|$ and $\|r_k^s\|/\|r_{ls}\|$. For
LSQR, recall from \cref{stopdesign} that $\|Sr_k^s\|$ monotonically decreases and converges to $\|Sr_s\|$
and thus
$$
\frac{\|r_k^s-r_{ls}\|}{\|r_{ls}\|}\rightarrow \frac{\|r_s-r_{ls}\|}{\|r_{ls}\|},
$$
which is bounded by $\sqrt{\frac{2\epsilon}{1-\epsilon}}$, as \Cref{thm:new} shows.
On the other hand, it is known from \eqref{eq: basic_inequality_c} that
$\|r_k^s\|$ overall decreases smoothly, so do $\frac{\|r_k^s-r_{ls}\|}{\|r_{ls}\|}$ and
$\frac{\|r_k^s\|}{\|r_{ls}\|}$, provided that $\epsilon$ is fairly small.
Bound \eqref{eq:basic foundation} indicates
that $\frac{\|r_k^s\|}{\|r_{ls}\|}$ ultimately stabilizes at one and is bounded from above by $\sqrt{\frac{1+\epsilon}{1-\epsilon}}$ at most. \Cref{fig:fig4,fig:fig5,fig:fig6} fully
confirm these theoretical results. Moreover, for each of the test problems,
by a careful comparison of \Cref{fig:fig10,fig:fig11,fig:fig12} and \Cref{fig:fig4,fig:fig5,fig:fig6},
we see that LSMR and LSQR stabilized approximately at the same iteration step, as is expected. That is,
LSMR and LSQR have solved \eqref{eq:original problem} by solving the sLS problem
\eqref{eq:sketched problem} and computed equally accurate optimal approximate solutions
of \eqref{eq:original problem}
when their own reliable stopping criteria \eqref{stopcrit1} and \eqref{stopcrit2} are used, respectively.
As shown in \Cref{fig:fig10,fig:fig11,fig:fig12}, when we use the LSMR method, the trend of  $\frac{\|r_k^s-r_{ls}\|}{\|r_{ls}\|}$ and
$\frac{\|r_k^s\|}{\|r_{ls}\|}$ still exhibit a smooth decline. However, for problems with extremely poor qualities, the curve may exhibit irregular decreasing behavior.

\Cref{fig:fig13,fig:fig14,fig:fig15} display the results with the
left figures corresponding to $d=1.2n$ and the right figures
corresponding to $d=2.4n$. All these results are obtained by using LSMR.

From \Cref{fig:fig13,fig:fig14,fig:fig15} and our recorded data, we find
that, when the row number $d$ of $S$ doubles, the corresponding relative residual norm decreases by
approximately a factor $\sqrt{2}$. These results justify the comments in \cref{subsec:compare}, \eqref{eq:normal_residual_upper_bound} and \Cref{thm:new}, which state that doubling $d$
approximately reduces
$\epsilon$ by $\sqrt{2}$ times.



\section{Conclusion}\label{sec:conclusions}

We have made a detailed and comprehensive numerical analysis on the sketch-and-solve paradigm.
In terms of the distortion $\epsilon$, we have
established the directional discrepancies $\|r_{ls}-r_s\|$,
and derived upper bounds for two metrics
that are the relative residual norm \eqref{eq:normal_residual_upper_bound} of the approximate solution
relative to the original LS problem \eqref{eq:original problem} and relative residual norm
\eqref{eq:corr_normal_residual_upper_bound} of the exact solution relative to the sLS problem
\eqref{eq:sketched problem}.
We have investigated backward errors of the approximate problems—a key metric for assessing the
stability of numerical approximations and established the sharp bounds for the minimal backward errors in
for the inconsistent and consistent LS problems. In the meantime, we have established error bounds
for the solution error $\|x_s-x_{ls}\|$ in terms of both the distortion $\epsilon$ and the residual
norm $\|r_{ls}\|$, and exposed an intrinsic relationship between them: For the same $\epsilon$,
the larger $\|r_{ls}\|$ is, the less accurate $x_s$ is as an
approximate solution of  the LS problem \eqref{eq:original problem}. These results reveal
the capabilities and limitations of the sketch-and-solve paradigm.

Based on the theoretical results and analysis, we have proposed two novel and general-purpose stopping criteria for iteratively solving the sLS problem to obtain an optimal approximate
solution of the original LS problem at the earliest iteration. They make an iterative solver such as
LSQR and LSMR terminate at right time. On the contrary to those traditional ones for a LS problem itself, the new stopping criteria make the iterations terminate much more early without sacrificing attainable accuracy of approximate solutions. Numerical experiments have confirmed all the theoretical results, and
justified the new stopping criteria and their separate safe use for LSQR and LSMR.


\section{Declarations}
The two authors declare that they have no conflict of interest,
and they read and approved the final manuscript.


\end{document}